\documentclass[reqno]{amsart}
\UseRawInputEncoding
\usepackage{tikz}
\usepackage{amsmath, amsfonts,amssymb,amsthm,amscd,latexsym,cite}
\usepackage{mathrsfs}
\usepackage{enumitem}
\usepackage{color}
\usepackage{verbatim}
\newtheorem{thm}{Theorem}[section]
\newtheorem{theorem}[thm]{Theorem}

\newtheorem{lemma}[thm]{Lemma}

\newtheorem{corollary}[thm]{Corollary}
\newtheorem{proposition}[thm]{Proposition}
\newtheorem{definition}[thm]{Definition}
\newtheorem{remark}[thm]{Remark}

\newtheorem{question}[thm]{Question}

\newcommand{\cA}{{\mathcal A}}
\newcommand{\cB}{{\mathcal B}}
\newcommand{\cC}{{\mathcal C}}
\newcommand{\cD}{{\mathcal D}}
\newcommand{\cE}{{\mathcal E}}
\newcommand{\cF}{{\mathcal F}}

\newcommand{\cH}{{\mathcal H}}

\newcommand{\cM}{{\mathcal M}}
\newcommand{\cN}{{\mathcal N}}

\newcommand{\cS}{{\mathcal S}}

\newcommand{\bC}{{\mathbb{C}}}

\setlist[enumerate,1]{label=\textup{(\roman*)}}

 \usepackage{color}
 \newcommand{\norm}[1]{\left\lVert#1\right\rVert}
\usepackage{hyperref}					
\hypersetup{colorlinks,
	linkcolor=blue,%
	citecolor=blue}

 \usepackage{float}
\begin{document}
\title[Disjointness-preserving mappings  and positive isometries]{Disjointness-preserving mappings   on 
Calkin operator spaces and positive isometries}

\author[K. Fang]{Kai Fang}\address[Kai Fang]{Institute for Advanced Study in Mathematics of HIT, Harbin 150001, China} \email{\color{blue}kaifang.8.25@gmail.com}

\author[J. Huang]{Jinghao Huang}\address[Jinghao Huang]{Institute for Advanced Study in Mathematics of HIT, Harbin 150001, China} \email{\color{blue}jinghao.huang@hit.edu.cn}

\author[K. Kudaybergenov]{Karimbergen Kudaybergenov}\address[K. Kudaybergenov]{Institute for Advanced Study in Mathematics of HIT, Harbin, 150001, China and
Suzhou Research Institute of HIT, Suzhou, 215104, China}
\email{\color{blue}kudaybergenovkk@gmail.com}

\author[R. Xu]{Ran Xu}\address[Ran Xu]{Institute for Advanced Study in Mathematics of HIT, Harbin 150001, China} \email{\color{blue}xxurann@stu.hit.edu.cn}

\thanks{\footnotesize
The work was supported by the NNSF of China  (No.12031004, 12301160 and 12471134)
and the 
Basic Research Program of Jiangsu (BK20251783).
 }

\subjclass[2020]{46B04; 46L52. \hfill 
}

\keywords{Calkin space; disjointness-preserving mapping;  positive isometry; strictly monotone norm}

\begin{abstract} 
  Let   
    $E(\cM,\tau)$  and  $F(\cM,\tau)$ be  two 
  Calkin operator spaces affiliated with  a  semifinite von Neumann algebra $\cM$  equipped with a semifinite faithful normal trace $\tau $. 
  We show that if $\cM$ is atomless, $\tau$ is finite, and $E(\cM,\tau)\not\subseteq F(\cM,\tau)$, then 
   every order-measure continuous and disjointness-preserving  mapping  $T:E(\cM,\tau)\xrightarrow{\rm into} F(\cM,\tau)$  is identical to  the zero mapping,
which establishes a noncommutative version of Abramovich's theorem.
We also show that every positive isometry $T$ from a normed $\cM$-bimodule $E(\cM,\tau)$  of $\tau$-measurable operators 
 into another $F(\cM,\tau)$ preserves disjointness provided that the norm of  $F(\cM,\tau)$ is  strictly monotone. 
As an application, we obtain the general form of $T$, which extends and unifies  several results due to Abramovich,
de Jager,  Conradie, 
Veksler and Sukochev et al.  \cite{SV,HSZ20,Abra1991,vek,dC20}. 

\end{abstract}
\maketitle

\section{Introduction}

\subsection{Background}
Disjointness-preserving mappings have been playing  an important  role in the theory of Banach lattices and symmetric function spaces.
We shall frequently omit below the adjective `linear' as we do not consider non-linear mappings in this paper.
In the classical  setting of lattices, any  regular disjointness-preserving operator
allows a multiplicative representation as a weighted composition operator, which
provides an abstract framework for a very important 
class of operators in analysis \cite{AK}. 
Abramovich\cite{Abra} obtained the multiplicative
representation of disjointness-preserving mappings on vector lattices.
Building on this,   Abramovich\cite{Abra1991} 
established a connection between the existence of identity embedding and the existence of an order continuous disjointness-preserving mapping
from a Calkin function space  (i.e., rearrangement invariant  ideal)  into another.
Precisely,   for (not necessarily normed) Calkin function spaces  $E(\Omega,\Sigma,\mu)$ and $F(\Omega,\Sigma,\mu)$ over an atomless finite measure space $(\Omega,\Sigma,\mu)$,  
if $$E(\Omega,\Sigma,\mu)\not\subseteq F(\Omega,\Sigma,\mu) ,$$ then every order continuous disjointness-preserving mapping
    $T:E(\Omega,\Sigma,\mu)\rightarrow F(\Omega,\Sigma,\mu)$ is identically equal to zero, i.e., $T\equiv0$\cite[Theorem 1]{Abra1991}.
An immediate  consequence of this result is the non-existence of a nontrivial disjointness-preserving operator from $L_p(\Omega,\Sigma,\mu)$ into $L_q(\Omega,\Sigma,\mu)$ for $0<p<q\leq \infty$ (see also \cite{Drahklin}).

\subsection{A noncommutative version of Abramovich's theorem concerning disjointness-preserving mappings}
Recall that  a  Calkin  space $E$ is a subspace of $c_0$ such that $a\in E$ and $\mu(b)\le \mu(a) $ implies $b\in E$, where $\mu(a)$ stands for the decreasing rearrangement. 
The classical Calkin correspondence states that the correspondence 
$E\leftrightarrow \cE$ is a bijection between Calkin sequence spaces and two-sided ideals of compact operators. 
The Calkin correspondence can be extended into the semifinite setting \cite[Section 2.4]{LSZ}.

Let $\cM$ be a (semi-)finite atomless (i.e., diffuse) von Neumann algebra equipped with a (semi-)finite faithful normal trace $\tau$. 
Let $S(\cM,\tau)$ be the $^*$-algebra of all $\tau$-measurable operators affiliated with $\cM$. 
A linear subspace $\cE(\cM,\tau)$  of  $S(\cM,\tau)$ is called a Calkin operator space if $a\in\cE(\cM,\tau)$ whenever $\mu(a)\le \mu(b)$ for some $b\in \cE(\cM,\tau)$\cite[Definition  2.4.1]{LSZ}. 
A Calkin function (respectively, sequence) space is the term reserved for a Calkin
operator space when $\cM=L_\infty (0, 1)$   or $\cM =L_\infty (0,\infty )$ (respectively, $\cM = \ell_\infty $).
The notation $\cS$ may refer to the space $S(0,\infty ),$ $S(0,1)$ or $\ell_\infty $, if the context is clear.

Let $\cM$ be an atomless (or atomic with all atoms having the same trace) von Neumann algebra equipped
with a faithful normal semifinite trace $\tau$. If $\cE (\cM, \tau)$ is a Calkin operator space, then
$$E :=\left\{ f\in  \cS  : \mu(f)=\mu(a) , ~a\in \cE(\cM,\tau)\right\}$$
is a Calkin function (or sequence) space. If $E$ is a Calkin function (or sequence) space,
then
$$\cE(\cM,\tau) :=\left\{ a\in S(\cM,\tau):  ~\mu(a)\in  E\right\}$$
is a Calkin operator space. This provides a canonical bijection between Calkin operator spaces and Calkin function (or sequence) spaces\cite[Theorem 2.4.4]{LSZ}. 
In the special case where $E$ is a  (Banach) symmetrically normed function space, the corresponding Calkin operator space is also 
 a (Banach) symmetrically normed operator space~\cite{LSZ,Kalton_S}.

It is natural to 
ask whether Abramovich's theorem \cite[Theorem 1]{Abra1991} holds in the noncommutative setting.  
For this purpose,    we make essential use of the multiplicative representation of order-local measure continuous disjointness-preserving mappings\footnote{A disjointness-preserving mapping on a noncommutative $L
_p$-space is  called a  separating mapping or a noncommutative Lamperti operator, see, e.g., \cite{MZ,HRW}. } (see Theorem \ref{positive}).
Indeed, in 2020, Sukochev, Zanin and the second author of the present paper established a multiplicative representation for positive disjointness-preserving mappings\cite[Theorem 3.1]{HSZ20}. 
Building on the approach used in \cite{HSZ20}, we obtain
  a noncommutative version of \cite[Theorem A]{Abra}, by showing that 
  a (not necessarily positive)  disjointness-preserving mapping
also admits a similar representation, where the mapping can be expressed as the composition of a partial isometry followed  by the multiplicative representation  of a positive mapping. 
Based on this, we establish a noncommutative version of Abramovich's theorem \cite[Theorem~1]{Abra1991}. 
\begin{theorem}\label{main1}
Assume that $\cM$ and $\cN$ are  two  atomless von Neumann algebras equipped with  faithful normal finite traces $\tau$ and  $\nu$, 
respectively, such that $\tau(\mathbf{1})=\nu(\mathbf{1})$.
If $E(0,\tau(\mathbf{1}))$ and $ F(0,\nu(\mathbf{1}))$  are  two Calkin function spaces satisfying  $E(0,\tau(\mathbf{1}))\not\subseteq F(0,\nu(\mathbf{1}))$, then 
   every disjointness-preserving and order-measure continuous (or normal) mapping 
    $T: E(\mathcal{M}, \tau)\xrightarrow{\rm into} F(\mathcal{N}, \nu) $ is 
    identically equal to zero, i.e., $T \equiv 0$.
\end{theorem}

The o-$t_{lm}$  continuity   of $T$  in Theorems \ref{main1} cannot be dispensed with.
On the other hand, the case of atomic von Neumann algebras is of no interest.
Detailed explanations  can be found in \cite[Section 5.2]{Abra1991}. 
Note that Theorem \ref{main1} does not hold
if the traces of von Neumann algebras are infinite, see, e.g., \cite[Section 5.4]{Abra1991}. 
A weak form of   Theorem \ref{main1} is given in  Theorem~\ref{infinite} below.

\subsection{A noncommutative version of Veksler's theorem concerning positive isometries}

The theory concerning the description of surjective isometries of commutative/noncommutative symmetric spaces has been extensively studied (see \cite{Kalton_Ran, z, FGHS,Huang_S,Yeadon,AC} and references therein).
However, the situation is more complicated when isometries
are not necessarily surjective.
In 1985, Veklser\cite{vek} obtained that if $E$ and $F$ are   function spaces and the norm of $F$ is strictly monotone (that is, $0\leq x_1\le x_2\in F$ and $x_1\ne x_2$ imply that $\left\|x_1\right\|_F<\left\|x_2\right\|_F$), then each positive isometry $T$ from $E$ into $F$ is disjointness-preserving. In particular, $T$ is of the elementary form given in  \cite{Abra}.
Note that 
the assumption  that the norm is strictly monotone  cannot be dispensed with, see, e.g.,  \cite[Example 4.7]{BHS}.

 Let $\cM$ and $\cN$ be two  semifinite von Neumann algebras with semifinite faithful normal traces $\tau$ and $\nu$, respectively.
Let  $E(\cM,\tau)$ and $F(\cN,\nu)$ be a normed $\cM$- and a normed $\cN$-bimodule, respectively.
One may  ask whether Veklser's result holds in the noncommutative setting.  
\begin{question}\label{nc v}
If $ F(\cN,\nu)$   has strictly monotone norm, then is every positive isometry from $E(\cM,\tau)$ into $F(\cN,\nu)$ necessarily disjointness-preserving?
\end{question}

Chilin et al.  \cite{CMS} studied this question in the  special setting of surjective isometries between two fully symmetric operator spaces (see also \cite{dC20}). 
In 2018, Sukochev and Veksler\cite{SV} gave an affirmative answer to the question above in the setting 
when  $\left\|\cdot\right\|_F$ is a strictly  monotone (fully symmetric) norm with respect to the Hardy--Littlewood--Polya submajorization.
In 2020,   Sukochev, Zanin and the second author of the present paper\cite{HSZ20} generalized this result to the setting when  $\norm{\cdot}_F$ is a strictly  monotone symmetric $\Delta$-norm with respect to $\log$-submajorization.
However,   
Question \ref{nc v} is not fully answered in the setting of general strictly monotone norms on $\cM$-bimodules.
The aforementioned results are based on the  condition $\mu(x+y)=\mu(x-y)$ for  $\tau$-compact operators $x,y$, which guarantees the orthogonality of $x$ and $y$ (i.e., $xy=0$)  \cite{SV,HSZ20}. 
However,  as demonstrated in  \cite[Example 4.7]{HSZ20},  $\mu(x+y)=\mu(x-y)$ does not imply $xy=0$ when 
 $0\le x,y\in S(\cM,\tau)$ are not  $\tau$-compact.  
Moreover, in order to   answer Question \ref{nc v} in  full generality for normed $\cM$-bimodules, we need to avoid using singular value functions. 
In Section~\ref{Sec:ortho}, we obtain a   necessary and sufficient condition for the orthogonality of  two positive measurable operators  (see Theorem \ref{ab=0} below), which strengthens the sufficient conditions given in \cite{SV,HSZ20} and is of interest in its own right. 
Having this at hand, we obtain the second main result of the present paper, which 
   extends and complements several  results in \cite{SV,HSZ20,vek,Abra,dC20} and answers Question \ref{nc v} in full generality.

\begin{theorem} \label{predisjoint}
Let $(\cM,\tau) $ and $ (\cN,\nu)$ be two semifinite von Neumann algebras with semifinite faithful normal traces $\tau$ and $\nu$, respectively.
Suppose that $E(\cM,\tau)$ and $F(\cN,\nu)$
are a normed $\cM$- and a norm $\cN$-bimodule, respectively.
If $\norm{\cdot}_F$ is strictly monotone, then every 
    positive isometry  $T$ from $E(\cM,\tau)$ into $F(\cN,\nu)$ is disjointness-preserving.
\end{theorem}

As an application of Theorem \ref{predisjoint},  when $(\cM,\tau)$ is finite,  we extend  the description of  positive isometries $T:E(\cM,\tau)\xrightarrow{\rm into}(\cN,\nu)$ given in \cite{SV,HSZ20} to the setting of strictly monotone norms (see Theorem~\ref{isoform} below), i.e., there exists a positive operator $b$ and a normal Jordan $^*$-monomorphism $J:\cM\to \cN$ such that 
$$
T(x)=bJ(x),\, x \in \cM.
$$
According to \cite[Proposition 8]{Abra1991}, the existence of a positive isometry from a symmetric function space $E$ into another  $F$ with a strictly monotone norm implies  $E\subseteq F$. 
In   Theorem~\ref{contain} below, we    establish a  noncommutative version of \cite[Proposition 8]{Abra1991}. 

At the International Conference on Banach Space Theory and its Applications at Kent, Ohio (August 1979), Pe\l czy\'nski posed the following question concerning the symmetric structure of ideals of compact operators on the Hilbert space $\ell_2$ (see also \cite[Question (B)]{Arazy81},  \cite[Problem A]{Arazy83} and \cite{HSS}):
{\it  
    Does the ideal $C_E$ of compact operators corresponding to an arbitrary separable symmetric sequence space $E$ have a unique symmetric structure?
}
One may consider an analogue of Pe\l czy\'nski's question in the sense of (positive) isometric isomorphisms on symmetric  operator spaces (see, e.g.,  \cite{FGHS} and \cite{Huang_S}) or, more generally,   on normed Calkin $\cM$-bimodules (see p. \pageref{definition normed Calkin}):
\begin{quote}
Let $\cM$ be an atomless semifinite von Neumann algebra equipped with a semifinite faithful normal trace $\tau$. 
If a symmetric operator space (more generally, a normed Calkin $\cM$-bimodule) $E(\cM,\tau)$ is isometric to another $F(\cM,\tau)$,
then does $E(0,\tau({\bf 1)}) $ coincide with $ F(0,\tau({\bf 1}))$ (as sets)?
\end{quote}
We refer to \cite{FGHS,Huang_S} for the 
  case of symmetric operator spaces. 
 Theorem \ref{contain}  gives an affirmative answer to this question for normed Calkin $\cM$-bimodules with strictly monotone norms in the sense of positive isometric isomorphisms. 

\section{Preliminaries}

In this section,
we recall some basic facts and notions which are needed for the proofs of the main results of this paper. 
General information on von Neumann
algebras and noncommutative symmetric  spaces  can be found in
\cite{LSZ,KPS,LT2}.
\subsection{$\tau$-measurable operators and  singular valued function }

Suppose that $a\in (0,\infty]$,  $I = (0,a)$ and $\Sigma$ is the $\sigma$-algebra of Lebesgue measurable subsets of~$I$. By $(I,m)$ we denote the measure space $(I,\Sigma,m)$ equipped with the
Lebesgue measure $m$. 
Let $S(I,m)$ (or $S(I)$ for brevity) be the space of all  (equivalent classes of)  finite Lebesgue measurable
real-valued functions on $I$. For $x\in S(I)$, we denote by
$\mu(x)$ the decreasing rearrangement of the function $|x|$~\cite{KPS,LSZ,LT2}. That is,
$$
\mu(t; x)=\inf \left\{s\geq0:\ m(\{|x|>s\})\leq t \right\},\quad t>0.
$$

Let $\mathcal{M}$ be a   von Neumann algebra on a   Hilbert space $\cH$.
Let ${\bf 1}$ be the identity.
Let $P(\mathcal{M})$ denote the lattice of all projections in $\mathcal{M}$ and
$U(\cM)$ denote the the set of all unitary elements in $\cM$.
The set of all self-adjoint elements in $\cM$ is denoted by $\cM_h$ and the set of all positive elements in $\cM$ is denoted by $\cM_+$.
For each self-adjoint operator $x$ affiliated with $\cM$,
we denote its spectral measure by $\{e^x\}$.
We say that a linear operator  $x$ is \emph{measurable} (denoted by $x \in S(\cM)$) if and only if $x$ is closed,
densely defined, affiliated with $\cM$, and $e 
^{|x|}
(\lambda,\infty )$ is a finite projection in $\cM$ for some $\lambda >0$. It follows
immediately that in the case when $\cM$ is a von Neumann algebra of type $III$ or a type $I$ factor,
we have $S(\cM, \tau) =\cM$. For type $II$ von Neumann algebras, this is no longer true\cite{LSZ,DPS}.

Let $\cM$ be a semifinite von Neumann algebra on a  Hilbert space $\cH$ equipped with a faithful normal semifinite trace $\tau$.
A measurable operator $x $ affiliated with $\cM$ is called \emph{$\tau$-measurable} if
$\tau\left(e ^{|x|}(\lambda,\infty)\right)<\infty$ for sufficiently large $\lambda$.
We denote the set of all $\tau$-measurable operators by
$S(\mathcal{M},\tau)$, which is a unital $^*$-algebra with respect to
strong sums and products (denoted simply by $x+y$ and $xy$ for all $x,y\in S(\cM,\tau)$)\cite{LSZ,DPS}.
The set of all self-adjoint elements in $S(\cM,\tau)$ is denoted by $S_h(\cM,\tau)$ and the set of all 
positive elements in $S_h(\cM,\tau)$ is denoted by $S(\cM,\tau)_+$.

For any closed and densely defined linear operator $x$, the null projection $n(x)=n(|x|)$ is the projection onto its kernel Ker$(x)$.
The left support projection $l(x)$ of $x$ is the projection onto the closure of its range Ran$(x)$ and the right support projection $r(x)$ of $x$ is defined by $r(x)=\mathbf{1}-n(x)$, which is the projection onto the closure of Ran$(x^*)$.

It is well known that if $x$ is a closed operator affiliated with $\cM$ with the polar decomposition $x=u|x|$, then $u\in U(\cM)$.
Moreover, $u^*u=l(x^*)=r(x)$ and $uu^*=l(x)=r(x^*)$\cite{DPS}.

The two-sided ideal $\mathcal{F}(\tau)$ in $\cM$ consisting of all elements with $\tau$- finite range projections is defined by
\[
\mathcal{F}(\tau)=\{x\in \cM: \tau(r(x))<\infty\}= \{x\in \cM: \tau(l(x))<\infty\}.
\]

For any $x \in S(\mathcal{M}, \tau)$,  the {\it spectral 
distribution function} of $|x|$ is defined by setting
\[
d_{|x|}(s) = \tau\bigl(e^{|x|}(s, \infty)\bigr), \quad s > 0.
\]
Note that $d_{|x|}$ is a right-continuous function (see, e.g., \cite{DPS}).
The {\it singular value function} is defined to be the right-continuous inverse of the spectral distribution function $d_{|x|}$, that is,
\[
\mu(t; x) = \inf\bigl\{ s \geq 0 : d_{|x|}(s) \leq t\bigr\}.
\]

\subsection{Symmetric spaces}\label{s:symmetric}

	Let $E(I,m)$ (or $E(I)$ for brevity) be a Calkin function space
	on $I$.
    If $E(I)$ is  equipped with a norm $\norm{\cdot}_E$ having property that
	 $\mu(x)\leq\mu(y)$ implies that 
        $\norm{x}_E\leq\norm{y}_E$ for arbitrary 
        $x, y\in E(I)$,
    then it is called a symmetrically normed function space.
    Moreover, if $E(I)$ is a Banach space, then it is called a {\it symmetric function space}.

	A linear subspace $\cE$ of $S(\mathcal{M},\tau)$ is called an {\it$\cM$-bimodule} of $\tau$-measurable operators (briefly, an  $\cM$-bimodule if no confusion may arise) if $uxv \in \cE$ whenever $x \in \cE$ and $u, v \in \cM$. 
    
    If $\cE$ is equipped with a (semi-)norm $\norm{\cdot}_{\cE}$ satisfying
	\[
	\norm{uxv}_{\cE} \leq \norm{u}_{L_\infty(\cM)}\norm{v}_{L_\infty(\cM)}\norm{x}_{\cE}, \quad x \in \cE, \, u, v \in \cM,
	\]
	then $\cE$ is called a (semi-)normed $\cM$-bimodule of $\tau$-measurable operators (briefly, a (semi-)normed $\cM$-bimodule).

	Let $E(\cM,\tau)$ be a Calkin operator space. 
If $E(\cM,\tau)$ is equipped with a norm $\norm{\cdot}_E$ satisfying $\mu(x)\leq\mu(y)$ implies that 
    $\norm{x}_E\leq\norm{y}_E$ for arbitrary 
    $x, y\in E(\cM,\tau)$, then it is called a symmetrically normed operator space. Moreover, if $E(\cM,\tau)$ is a Banach space, then it is called a symmetric operator space.

It should be mentioned that an arbitrary  Calkin operator space  $  E(\cM,\tau)\subseteq S(\mathcal{M},\tau)$ is an $\cM$-bimodule and an arbitrary symmetrically normed operator space is automatically a normed $\cM$-bimodule \cite[Proposition 4.4.3]{DPS}.
 
Each normed $\cM$-bimodule $\cE$ is an {\it absolutely solid space} of $S(\mathcal{M},\tau)$ in the following sense.
	If $x \in S(\mathcal{M},\tau)$ and $y \in \cE$ satisfy $|x| \leq |y|$ in $S(\cM, \tau)$, then $x \in \cE$\cite[Proposition 4.1.3]{DPS}.



Let $(E,\left\|\cdot\right\|_{E})$ be a symmetrically normed   function   space on $(0,\tau({\bf 1}))$. The operator space $E(\cM,\tau)$ defined by
\begin{align*}
    E(\mathcal{M},\tau):=\Big\{x\in S(\mathcal{M},\tau):\ \mu(x)\in E\Big\},
    \quad
    \norm{x}_{E (\mathcal{M},\tau)}:=\norm{\mu(x)}_E
\end{align*}
is naturally  a symmetrically normed operator space. We note that $E(\mathcal{M},\tau)$  is a Banach space with respect to $\left\|\cdot\right\|_{E (\mathcal{M},\tau)}$ if $E$ is Banach, and it is called the (non-commutative) symmetric operator space associated
with $(\mathcal{M},\tau)$ corresponding to $(E,\left\|\cdot\right\|_{E})$\cite{Kalton_S, LSZ}. 
When $E=L_p(0, \infty)$,   
$1 \leq p \leq \infty$, the corresponding noncommutative space $E(\cM, \tau) = L_p( \cM, \tau)$ is the non-commutative $L_p$ space. 

\subsection{Measure topology and order convergence}\label{subsec 2}

For convenience of the reader, we also recall the definition of the measure topology $t_m$ on the algebra $S(\cM,\tau)$. For every $\varepsilon,\delta>0$, we define the neighborhood
\[
V(\varepsilon,\delta)=\{x\in S(\cM,\tau):\exists  p\in P(\cM) \,\text{such that} \, \norm{x(\mathbf{1}-p)}_{L_\infty(\cM)}\leq \varepsilon, \tau(p)\leq \delta\}.
\]
The collection $\{V(\varepsilon,\delta):\varepsilon,\delta>0\}$ is a neighborhood base at zero  for a complete metrizable Hausdorff vector space topology $t_m$ on $S(\cM,\tau)$  (see, e.g., \cite{DPS}).
If a net $\{x_i\}_{i\in I}$ in $S(\cM,\tau)$ converges to the operator $x\in S(\cM,\tau)$ in $t_m$,
then this is denoted by $x_i\xrightarrow{t_m}x$ and the net $\{x_i\}_{i\in I}$ is said to converge to $x$ in measure.


The measure topology can also be characterized in terms of the singular value function.
For a net $\{x_i\}_{i\in I} \subset S(\cM,\tau)$, 
$x_i \xrightarrow{t_m}  0$ if and only if  
\begin{align}\label{m-mu}
    \mu(t;x_i )\to  0,\quad   t>0.
\end{align}

Recall the definition of local measure topology. 
For every $\varepsilon,\delta>0$, $e\in P(\cM)$ with $\tau(e)<\infty$, we define 
\[
V(\varepsilon,\delta,e)=\{x\in S(\cM):exe\in V(\varepsilon,\delta)\}.
\]
The collection $\{V(\varepsilon,\delta,e)\}$ forms a neighborhood base at zero for a Hausdorff vector space topology $t_{lm}$ in $S(\cM)$ (see \cite[Proposition 2.7.4]{DPS}).
Observe that a net $\{x_i\}_{i\in I}\subset S(\cM)$ converges to $x\in S(\cM)$ in the local measure topology (denoted by $x_i\xrightarrow{t_{lm}}x$ and said to converge locally in measure),
if and only if $ex_i e\xrightarrow{t_{m}}exe$ for every $e\in P(\cM)$ with $\tau(e)<\infty$.

The following result is essentially known to experts. However, due to the lack of suitable references, we provide a complete proof below. 
\begin{lemma}\label{lm}
Let   $\{x_i\}_{i\in I}$ be a net in $S(\cM)$.
If $|x_i|\xrightarrow{t_{lm}} 0 $, then  $x_i\xrightarrow{t_{lm}}0 $.
\end{lemma}

\begin{proof}
    For each $e\in P(\cM)$ with $\tau(e)<\infty$, we have 
    $$\big||x_i|^{\frac{1}{2}}e\big|^2=e|x_i|e\xrightarrow{t_{m}}0.$$
    It follows from \cite[Proposition 3.2.11(i) and Proposition 3.2.8]{DPS} and \eqref{m-mu} that  for each $t>0$, 
    \begin{align}\label{mu1}
        \mu\left(t;|x_i|^{\frac{1}{2}}e\right)=\mu\left(t;\big||x_i|^{\frac{1}{2}}e\big|^2\right)^{\frac{1}{2}}\to 0.
    \end{align}
Similarly,  for each $t>0$,  we have 
\begin{align}\label{mu2}
    \mu\left(  t; e|x_i^*|^{\frac{1}{2}} \right)\to  0.
\end{align}
Let  $x_i=u_i|x_i|$ be the polar decomposition. 
Observe that $ u_i | x_i|  u_i^* = |x_i^*| $  (see, e.g., \cite[p. 19]{DPS}). 
    We have $ u_i | x_i|^{1/2}  u_i^* =(u_i | x_i|  u_i^*)^{1/2}= |x_i^*|^{1/2}$. 
    
    For each $t>0$, it follows from \cite[Proposition 3.2.7(iv)(vi)]{DPS} that
    \begin{align*}
    \mu(2t;ex_i e)&\leq  \mu(t; e u_i|x_i|^{\frac{1}{2}})\mu(t;|x_i|^{\frac{1}{2}}e)=\mu(t; e |x_i^*|^{\frac{1}{2}}u_i )\mu(t;|x_i|^{\frac{1}{2}}e)\\
    &\le \mu(t; e |x_i^*|^{\frac{1}{2}})\mu(t;|x_i|^{\frac{1}{2}}e)
    \xrightarrow{\eqref{mu1}\eqref{mu2}}  0,
    \end{align*}
    which implies that $ex_i e\xrightarrow{t_{m}}0$ (see \eqref{m-mu}).
    Since $e$ is arbitrary, it follows from the definition of the local measure topology that 
    \[
    x_i\xrightarrow{t_{lm}} 0.
    \]
    This completes the proof. 
\end{proof}
\begin{remark}
The converse of the above lemma does not hold. 
That is, there exists a net $\{x_i\}_{i\in I} \subset S(\cM,\tau)$ converging locally in measure to zero but we do not have 
$|x_i| \xrightarrow{t_{lm}} 0$, see a concrete  example  in \cite[Remark 2.7.8]{DPS}.
\end{remark}
 
In the classical setting, a net $\{x_\alpha\}_{\alpha\in A}$ in a partially ordered set is said to be order convergent to $x$ (denoted by $x_\alpha\xrightarrow{(o)}x$), whenever there exist two nets $\{y_\beta\}_{\beta\in B}$ and $\{z_\gamma\}_{\gamma\in \Gamma}$ such that:
\begin{enumerate}
    \item $y_\beta\uparrow x$ and $z_\gamma \downarrow x$.
    \item For each $\beta \in B$ and $\gamma \in \Gamma$, there exist some $\alpha_0\in A$ satisfying $y_\beta \leq x_\alpha \leq z_\gamma$ for all $\alpha>\alpha_0$.
\end{enumerate}
For an order bounded net $\{x_\alpha\}_{\alpha\in A}\subset E$ in a Dedekind complete Riesz space $E$, $x_\alpha\xrightarrow{(o)}x$ in $E$ if and only if there exists a net $\{u_\alpha\}_{\alpha\in A}$ such that $u_\alpha\downarrow 0$ and $|x_\alpha-x|\leq u_\alpha$ for each $\alpha \in A$ (see \cite[Lemma 1.19]{AA}).

 Let $(\cM,\tau)$ be a semifinite von Neumann algebra with a semifinite faithful normal trace $\tau$.
We define order convergence as follows.
\begin{definition}
Let $\cE$ be an $\cM$-bimodule.
    A net $\{x_i\}_{i\in I}\subset \cE$ is said to be order convergent to $x$ (denoted by $x_i\xrightarrow{(o)}x$) whenever there exists a net $\{u_i\}_{i\in I}\subset \cE_+$ such that $u_i\downarrow 0$ and $|x_i-x|\leq u_i$ for each $i \in I$.
\end{definition}

By \cite[Proposition 2.7.6 (iv)(v)]{DPS} and Lemma \ref{lm}, we have the following proposition, which implies that the limit of an order convergent net is unique.

\begin{proposition}\label{o-lm}
    Let  $\cE$ be an $\cM$-bimodule. For any net $\{x_i\}\subset \cE$,  $x_i\xrightarrow{(o)}x$ implies $x_i\xrightarrow{t_{lm}}x$.
\end{proposition}

Let $(\cM,\tau) $ and $ (\cN,\nu)$ be two semifinite von Neumann algebras with semifinite faithful normal traces $\tau$ and $\nu$, respectively.
Suppose that $E(\cM,\tau)$ and $F(\cN,\nu)$
are an $\cM$- and an $\cN$-bimodule, respectively. 
$T:E(\cM,\tau)\xrightarrow{\rm into} F(\cN,\nu)$ is a linear mapping.
    \begin{enumerate}
        \item $T$ is said to be  order continuous if $x_i\xrightarrow{(o)}x$ implies $T(x_i)\xrightarrow{(o)}T(x)$ for any net $\{x_i\}_{i\in I}\subset E(\cM,\tau)$.
        \item $T$ is said to be  normal if $T(x)\ge 0$  for all $0\le x\in E(\cM,\tau)$ and  $x_i\uparrow x$ implies  $T(x_i)\uparrow T(x)$ for any net $\{x_i\}_{i\in I}\subset E(\cM,\tau)_+$.
        \item For an arbitrary net $\{x_i\}_{i\in I}\subset E(\cM,\tau)$ satisfying $x_i\xrightarrow{(o)}x\in E(\cM,\tau)$, 
        if $T(x_i)\xrightarrow{t_{m}} T(x)$ (respectively, $T(x_i)\xrightarrow{t_{lm}} T(x)$), then $T$ is said to be order-measure continuous (respectively,  order-local measure continuous), or 
     o-$t_{m}$ continuous 
          (respectively, o-$t_{lm}$ continuous) for short.
    \end{enumerate}

It follows from Proposition \ref{o-lm} that an order (or local measure) continuous mapping is necessarily o-$t_{lm}$ continuous.
By \cite[Proposition 2.7.6]{DPS}, it is readily verified that any positive o-$t_{lm}$ continuous mapping is necessarily normal.

\subsection{Jordan $^*$-homomorphisms}\label{s:Jordan}

Let $(\cM,\tau) $ and $ (\cN,\nu)$ be two semifinite von Neumann algebras with semifinite faithful normal traces $\tau$ and $\nu$, respectively.
A complex-linear mapping $J:\cM\to\cN$ (respectively, $S(\cM,\tau)\to S(\cN,\nu)$)
is called a {\it Jordan $^*$-homomorphism} if 
\[
J(x^2) = J(x)^2 \, \text{and} \, J(x^*) = J(x)^* ,\, x \in \cM\,\text{(respectively, $S(\cM,\tau)$)}.
\]

We call $J$ a \textit{Jordan $^*$-monomorphism} if it is injective,
and a \textit{Jordan $^*$-isomorphism} if it is bijective.
Further details regarding von Neumann algebras and Jordan homomorphisms may be
found in \cite{KR, BraRobinbook}.

\begin{proposition}[{see, e.g., \cite[Proposition 2.14]{HSZ20}}]\label{Jcom}
If $J$ is a Jordan $^*$-homomorphism from $\cM$ into $\cN$, then for any commuting $x,y\in \cM$, we have 
\begin{align*}
J(xy)=J(x)J(y)=J(y)J(x).
 \end{align*}
\end{proposition}

If $J$ is a Jordan $^*$-homomorphism, then for any self-adjoint $x\in \cM$, $J(x)$ is self-adjoint. It is well-known \cite[p. 211]{BraRobinbook} that every Jordan $^*$-homomorphism is positive, i.e., if $x\geq0$, then $J(x)\geq0$.

If $J:\cM\rightarrow\cN$ is a Jordan $^*$-isomorphism, then $J$ is necessarily normal\cite[Appendix A]{RR}.

\begin{remark}[{see, e.g., \cite[Remark 2.16]{HSZ20}}]\label{Jsubalg}
    Assume that $J:\cM\rightarrow\cN$ is a normal Jordan $^*$-homomorphism.  
Then   $J(\cM)$ is a von Neumann subalgebra  of 
$J(\mathbf{1})\cN J(\mathbf{1})$. 
In particular, if $J$ is   injective,
then $J$ is a normal Jordan $^*$-isomorphism from $\cM$ onto $J(\cM)$.
\end{remark}

The following proposition provides a characterization of the continuity of Jordan $^*$-homomorphisms in the measure topology.
\begin{proposition}[see, e.g., {\cite[Theorem 3.13]{weigt}}{\cite[Proposition 4.7(i)]{Lab}}]\label{Jmeasure}
     If $J:\cD(J)\rightarrow\cN$ is a Jordan $^*$-homomorphism such that its domain $\cD(J)$ is a $^*$-subalgebra of $\cM$ containing all projections of finite trace,
then the following statements are equivalent.
\begin{enumerate}
    \item $J$ is continuous in measure topology.
    \item For each $\varepsilon>0$, there exists some $\delta>0$ such that $\nu(J(p))\leq\varepsilon$ whenever $p\in P(\cM)$ and $\tau(p)<\delta$.
\end{enumerate}
\end{proposition}

Using this proposition, we obtain the following result.

\begin{proposition}\label{Jnormea}
Let $\cM$ and $\cN$ be two von Neumann algebras equipped with finite faithful normal trace $\tau$ and $\nu$, respectively. 
If $J:\cM\to \cN$ is a Jordan $^*$-homomorphism, then $J$ is normal if and only if $J$ is continuous in measure topology. 
\end{proposition}
\begin{proof}
Assume that $J$ is continuous in measure topology. 
By \cite[Theorem 2.6.3 and Proposition~2.6.1(iii)]{DPS}, we have $J$ is normal.

Assume that $J$ is normal.
Let $\{p_n\}_{n\ge1}$ be an arbitrary  sequence  in  $ P(\cM)$.
We claim that $\tau(p_n)\rightarrow 0$ as $n\to \infty $ implies $\nu(J(p_n))\rightarrow 0$  as $n\to \infty $.
Indeed, 
assume that there exists a subsequence $\{p_{n_k}\}_{k\ge 1} \subset\{p_n\}_{n\ge1} $ and a constant $\varepsilon_0$ such that $\nu(J(p_{n_k}))>\varepsilon_0$ for each $ k\in \mathbb{N}$.
Without loss of generality, we may   denote the subsequence by $\{p_n\}_{n\ge 1}$ and assume that $\tau(p_n)\leq \frac{1}{2^n}$. Define 
\[
q_n:=\bigvee_{k=n}^\infty p_k.
\]
Noting that $q_n\downarrow$ and $\tau(q_n)\leq \sum\limits_{k=n}^\infty\frac{1}{2^n}=\frac{1}{2^{n-1}}$, we have $q_n\downarrow0$ as $n\to \infty$.
Since $J$ is normal, it follows that $J(q_n)\downarrow0$, which together with $\nu(J(\mathbf{1}))<\infty$ yields that
$$\nu(J(q_n))\downarrow0.$$
However, $\nu(J(q_n))\geq \nu(J(p_n))>\varepsilon_0$, which is a contradiction. This proves our claim.
By Proposition \ref{Jmeasure}, we have $J$ is continuous in measure topology.

\end{proof}

It is well known that $\cM$ is dense in $S(\cM,\tau)$ with respect to measure topology\cite[Proposition 2.5.4]{DPS}. Hence, if $J:\cM\rightarrow \cN$ is a unital Jordan $^*$-homomorphism which is continuous in the measure topology, then  $J$ can be uniquely extended to a Jordan $^*$-homomorphism $\hat{J}:S(\cM,\tau)\rightarrow S(\cN,\nu)$.
If $J$ is a Jordan  $^*$-isomorphism, then so is $\hat{J}$ (see, e.g., the proofs in \cite[Proposition 2.15]{HSZ20} and \cite[Proposition 3.4(2)]{BHS}).
If $J$ is normal, then $\hat{J}$ is normal\cite[Proposition 3.3]{BHS}.

\section{The elementary form of order-local measure continuous disjointness-preserving mappings}

The main purpose of this section is to establish a general description of o-$t_{lm}$ continuous and disjointness-preserving mappings on $\cM$-bimodules of $\tau$-measurable operators. 

Let $(\cM,\tau) $ and $ (\cN,\nu)$ be two semifinite von Neumann algebras with semifinite faithful normal traces $\tau$ and $\nu$, respectively.
Suppose that $E(\cM,\tau)$ and $F(\cN,\nu)$
are an $\cM$- and an $\cN$-bimodule, respectively. 
A linear mapping $T:E(\cM,\tau)\xrightarrow{\rm into} F(\cN,\nu)$ is said to be  disjointness-preserving if $$T(x)^*T(y)=T(x)T(y)^*=0$$ whenever $x,y\in E(\cM,\tau)$ with $x^*y=xy^*=0$. In terms of  left and right support projections, the mapping $T$ is disjointness-preserving if and only if it preserves the disjointness of left and right support projections, i.e., $l(T(x))l(T(y))=r(T(x))r(T(y))=0$ whenever $x,y\in E(\cM,\tau)$ with $l(x)l(y)=r(x)r(y)=0$.

We now present a  generalization of the representation in \cite[Theorem 3.1]{HSZ20} to $\cM$-bimodules.

\begin{theorem}\label{normal}
    Suppose that 
    $(\cM,\tau)$ is a von Neumann algebra with a finite faithful normal trace $\tau$ and $(\cN,\nu)$ is a semifinite von Neumann algebra with a semifinite faithful normal trace $\nu$.
    Let $E(\mathcal{M}, \tau)$ and $F(\cN,\nu)$ be an 
    $\cM$- and  an $\cN$-bimodule, respectively.
    If $T: E(\cM, \tau)\xrightarrow{\rm into} F(\cN, \nu)$ 
     is positive and disjointness-preserving, then there exists a Jordan $^*$-homomorphism $J:\cM\to \cN$ such that for each $x\in \cM$,
     \[
     T(x)=T(\mathbf{1})J(x)
     \]
     and $J(x)$ commutes with $T(\mathbf{1})$.
If, in addition, $\nu(J(\mathbf{1}))<\infty$ and $T$ is normal, then $J$ can be extended to a normal Jordan $^*$-homomorphism from $S(\cM,\tau)$ into $S(\cN,\nu)$ such that for each $x\in E(\cM,\tau)$,
    \[
    T(x)=T(\mathbf{1})J(x)
    \]
    and $T(\mathbf{1})$ commutes with $J(x)$.
\end{theorem}

\begin{proof}
    Arguing  mutatis mutandis as in the  proof of \cite[Theorem 3.1]{HSZ20},
    there exists a Jordan $^*$-homomorphism $J:\cM\to \cN$ such that for each $x\in \cM$,
     \[
     T(x)=T(\mathbf{1})J(x)
     \]
     and $J(x)$ commutes with $T(\mathbf{1})$. 
In particular, if $T$ is normal, then $J$ is normal.     

Consider the case when $\nu(J(\mathbf{1}))<\infty$ and $T$ is normal.
 By Remark \ref{Jsubalg} and Proposition \ref{Jnormea},  $J$ is continuous in the measure topology. Since $\cM$ is dense in $S(\cM,\tau)$ with respect to measure topology, it follows that
$J$ can be uniquely extended to a normal Jordan $^*$-homomorphism
from $S(\cM,\tau)$ into $S(J({\bf 1}) \cN J({\bf 1}) ,\nu)\subset S(\cN,\nu)$ (see  Section \ref{s:Jordan}).

For each $x\in E(\cM,\tau)$, observing that $x$ is a linear combination of four positive elements in $E(\cM,\tau)_+$, to prove the second  statement of the theorem, without loss of generality, we may assume that $x\in E(\cM,\tau)_+$. 
There exists a net $\{x_i\}_{i\in I}\subset \cM_+$ such that $x_i\uparrow x$, which together with \cite[Proposition 2.6.3]{DPS} yields that $x_i\xrightarrow{t_{m}} x$. 
It follows from the continuity of $J$ in measure that $J(x_i)\xrightarrow{t_{m}} J(x)$. 
By \cite[Proposition 2.6.11]{DPS}, we have
\begin{align}\label{meacon}
T(x_i)=T(\mathbf{1})J(x_i)\xrightarrow{t_m}T(\mathbf{1})J(x).
\end{align}
On the other hand, since  $T$ is normal, it follows from 
  \cite[Theorem 2.6.3]{DPS} that 
\[
T(x_i)\xrightarrow{t_{m}}T(x),
\]
which together with \eqref{meacon} yields that 
\[
T(x)=T(\mathbf{1})J(x),\quad  x\in E(\cM,\tau).
\]
Observing that $T(\mathbf{1})J(x)= J(x ) T(\mathbf{1})$ for all $x\in \cM$, it follows from the continuity of $J$ in measure that  $T(\mathbf{1})$ commutes with $J(x)$ for each $x\in E(\cM,\tau)$.
\end{proof}

Using Theorem \ref{normal}, we establish a general representation of o-$t_{lm}$ continuous disjointness-preserving mappings on $\cM$-bimodules.

\begin{theorem}\label{positive}
    Suppose that 
    $(\cM,\tau)$ is a von Neumann algebra with a finite faithful normal trace $\tau$ and $(\cN,\nu)$ is a semifinite von Neumann algebra with a semifinite faithful normal trace $\nu$.
    Let $E(\mathcal{M}, \tau)$ and $F(\cN,\nu)$ be an
    $\cM$- and an $\cN$-bimodule, respectively.
    If
     $T: E(\cM, \tau)\xrightarrow{\rm into} F(\cN, \nu)$ 
     is o-$t_{lm}$ continuous  and disjointness-preserving, 
     then there exists a normal Jordan $^*$-homomorphism $J:\cM\to \cN$ and a partial isometry $w$ such that
     \[
     T(x)=w|T(\mathbf{1})|J(x), \quad  x\in \cM.
    \]
    Moreover, $|T(\mathbf{1})|$ commutes with $J(x)$ for each $x\in \cM$.

    In particular, if $\nu(J(\mathbf{1}))<\infty$, then $J$ can be extended to a normal Jordan $^*$-homomorphism from $S(\cM,\tau)$ into $S(\cN,\nu)$ such that
    \[
    T(x)=w|T(\mathbf{1})|J(x), \quad  x\in E(\cM,\tau)
    \]
     and $|T(\mathbf{1})|$ commutes with $J(x)$ for each $x\in E(\cM,\tau)$.
\end{theorem}

\begin{proof}
We only need to consider the situation when $T\not\equiv0$.
For each projection $e\in P(\cM)$, 
we denote the polar decomposition of 
$T(e)$ by
$T(e)=u_e|T(e)|$. In particular, $T(\mathbf{1})=u_\mathbf{1}|T(\mathbf{1})|$.
Suppose that  $p,q\in P(\cM)$ satisfies $pq=0$.
Since $T$ is disjointness-preserving, it follows that
$r(T(q))r(T(p))=l(T(q))l(T(p))=0$.
In particular, we have 
\begin{equation}\label{modu}
\begin{aligned}
&|T(p+q)|
=\left(\left(T(p)^*+T(q)^*\right)\left(T(p)+T(q)\right)\right)^{\frac{1}{2}}
=\left(|T(p)|^2+|T(q)|^2\right)^{\frac{1}{2}}\\
=&\left(|T(p)|^2+\left|T(p)\right|\left|T(q)\right|+\left|T(q)\right|\left|T(p)\right|+|T(q)|^2\right)^{\frac{1}{2}}
=|T(p)|+|T(q)|.
\end{aligned}
\end{equation}
Noting  that
\[
r(u_p)=l(|T(p)|)=r(T(p)),\quad l(u_p)=l(T(p)),
\]
\[
r(u_q)=l(|T(q)|)=r(T(q)),\quad l(u_q)=l(T(q)),
\]
we have $u_p|T(q)|=u_q|T(p)|=0$.
Thus,
\begin{align*}
T(p+q)&=T(p)+T(q)=u_p|T(p)|+u_q|T(q)|\\
&=(u_p+u_q)(|T(p)|+|T(q)|)
\stackrel{\eqref{modu}}{=}(u_p+u_q) |T(p+q)|.
\end{align*}
Moreover, since 
\[
r\left(u_p\right)r(u_q)=l(u_p)l(u_q)=r(|T(p)|)r(|T(q)|)=l(|T(p)|)l(|T(q)|)=0,
\]
it follows that  $u_p+u_q$ is a partial isometry with 
$r(u_p+u_q)=r(u_p)+r(u_q)=l(|T(p)|)+l(|T(q)|)=l(|T(p)|+|T(q)|)$.
By the uniqueness of polar decomposition (see, e.g., \cite[Theorem 1.7.3]{DPS}),
we have
\begin{align*}
u_{p+q}=u_p+u_q  .
\end{align*}
Hence, 
given $a,b\in \mathbb{R}$, by \eqref{modu}, we have
\begin{align}\label{linear}
T(ap+bq)=u_{p+q}(|aT(p)|+|bT(q)|),\quad|T(ap+bq)|=|aT(p)|+|bT(q)|.
\end{align}
Therefore,
\begin{align}\label{T(p)}
    T(p)=T(p+0(\mathbf{1}-p))
    \stackrel{\eqref{linear}}{=}u_\mathbf{1}|T(p)|,
\quad  p\in P(\cM).
\end{align}

For each $0\leq x\in E(\cM,\tau)$, define $x_n:=\sum\limits_{k=1}^{n^2}\frac{k-1}{n}e^x(\frac{k-1}{n},\frac{k}{n}]$.
Since
\begin{align*}
T(x_n)&=\sum\limits_{k=1}^{n^2} \frac{k-1}{n}T\left(e^x\left(\frac{k-1}{n},\frac{k}{n}\right]\right)\stackrel{\eqref{T(p)}}{=}
   \sum\limits_{k=1}^{n^2} \frac{k-1}{n}u_\mathbf{1}\left|T\left(e^x\left(\frac{k-1}{n},\frac{k}{n}\right]\right)\right|
 \\& \stackrel{\eqref{linear}}{=}
    u_\mathbf{1}\left|T\left(\sum\limits_{k=1}^{n^2} \frac{k-1}{n}e^x\left(\frac{k-1}{n},\frac{k}{n}\right]\right)\right|
=u_\mathbf{1}|T(x_n)|,    
\end{align*}
it follows that $u_\mathbf{1}^*T(x_n)\ge 0$.

Noting that $x_n\xrightarrow{(o)}x$ as $n\to \infty $,
it follows from the o-$t_{lm}$ continuity of $T$ that
$T(x_n)\xrightarrow{t_{lm}}T(x)$,
which together with \cite[Proposition 2.7.5 and Proposition 2.7.6(i)]{DPS} yields that 
\begin{align}\label{Tlmcon}
    u_\mathbf{1}^*T(x_n)\xrightarrow{t_{lm}}u_\mathbf{1}^*T(x)\ge 0.
\end{align}
Hence, $u_\mathbf{1}^*T(\cdot)$ is    disjointness-preserving and positive on $E(\cM,\tau)$.
Moreover, it is o-$t_{lm}$ continuous, which implies that 
$u_\mathbf{1}^*T(\cdot)$ is normal (see  Section \ref{subsec 2}).
By Theorem~\ref{normal}, the proof is complete.
\end{proof}

If $T$ in Theorem \ref{positive} is not necessarily o-$t_{lm}$ continuous, but $E(\cM,\tau)$ has order continuous norm (i.e., $\norm{x_\alpha}_E\downarrow0$ whenever $0\leq x_\alpha\downarrow0$ in $ E(\cM,\tau)_+$), then we have a similar result. 

\begin{proposition}\label{ocnorm}
    Suppose that 
    $(\cM,\tau)$ is a von Neumann algebra with a finite faithful normal trace $\tau$ and $(\cN,\nu)$ is a semifinite von Neumann algebra with a semifinite faithful normal trace $\nu$.
    Let $E(\mathcal{M}, \tau)$ and $F(\cN,\nu)$ be a 
    normed $\cM$- and a normed $\cN$-bimodule, respectively.
     If $\norm{\cdot}_E$ is order continuous  and $T:E(\mathcal{M}, \tau)\xrightarrow{\rm into} F(\cN,\nu)$ is a disjointness-preserving bounded mapping, 
    then $T$ is as form in Theorem \ref{positive}.
\end{proposition}

\begin{proof}
For each $0\leq x\in E(\cM,\tau)$, define $x_n:=\sum\limits_{k=1}^{n^2}\frac{k-1}{n}e^x(\frac{k-1}{n},\frac{k}{n}]$.
By the proof of Theorem \ref{positive}, we have
$T(x_n)=w|T(x_n)|$,
i.e., $w^*T(x_n)\ge0$.

Since $E(\mathcal{M}, \tau)$ has order continuous norm, it follows from $x_n\uparrow x$ that $\norm{x-x_n}_E\to 0$ as $n\to \infty $.
Hence,
\begin{align}\label{normcon}
    \norm{w^*T(x-x_n)}_F\le \norm{T}\norm{x-x_n}_E\to 0,
\end{align}
which together with \cite[Corollary 4.1.16(i)]{DPS} yields that $w^*T(x)$ is positive.
Therefore, $w^*T(\cdot)$ is a positive disjointness-preserving bounded mapping from $E(\cM,\tau)$ into $F(\cN,\nu)$.

For any $0\le x_\alpha \uparrow x\in E(\cM,\tau)$, 
arguing similarly as \eqref{normcon}, we have
$w^*T(x_\alpha)\uparrow w^*T(x)\in F(\cN,\nu)$\cite[Propositions 2.6.1(iii) and 4.4.4]{DPS}.
Consequently,
$w^*T(\cdot)$ is normal,
which together with Theorem \ref{normal} completes the proof.
\end{proof}

In what follows, 
    unless stated otherwise, we assume that $\cM$ is a semifinite von Neumann algebra equipped with a semifinite faithful normal trace $\tau$.

To establish a general representation of disjointness-preserving mappings on $\cM$-bimodules, we need to employ the following lemma,
which follows from an argument similar to that used in   \cite[Proposition 2.17]{HSZ20}.
For the sake of completeness, we include a full proof. 

\begin{lemma}\label{solim}
    Let $\{x_i\}\subset \cM$ be a uniformly bounded net  and $\{p_i\}$ (respectively $\{q_i\}$) is a net of projections increasing to $p\in P(\cM)$ (respectively, $q\in P(\cM)$). If $x_i=q_ix_jp_i$ for every $j\geq i$, then the strong operator limit $so-\lim\limits_{i} x_i$ exists (denoted by $x$). In particular, $x_i=q_ixp_i$  and $x=qxp\in \cM$.
\end{lemma}

\begin{proof}
    Since the unit ball of a von Neumann algebra is compact with respect to weak operator topology (see, e.g., \cite[Chapter IX, Proposition 5.5]{Conway}) and $\{x_i\}$ is uniformly bounded, it follows that there exists a subnet $\{x_{i_k}\}$ of $\{x_i\}$ such that 
     the weak operator limit
    $wo-\lim_{k}x_{i_k}\in \cM$ exists. Let 
    \begin{align}\label{xdefine}
        x:=wo-\lim_{k}x_{i_k}.
    \end{align}
    In particular, $x_{i_k}=wo-\lim\limits_{i_j\geq i_k} q_{i_k} x_{i_j}p_{i_k}=q_{i_k}xp_{i_k}.$
    Hence, for each $i\leq i_k$, we have
    \[
    x_i=q_ix_{i_k}p_i=q_iq_{i_k}xp_{i_k}p_i=q_ixp_i.
    \]
    Since $\{x_{i_k}\}$ is a subnet of $\{x_i\}$, it follows that $x_i=q_ixp_i$ for every $i$. 

    Since multiplication is jointly so-continuous when 
restricted to norm bounded sets (see \cite[p. 2]{DPS}),
we have   $x_i=q_ixp_i\to qxp $ in the strong operator topology.
    This  together with \eqref{xdefine} yields $x=qxp$.
\end{proof}




Let $\cM$ and $\tau$ in Theorem \ref{positive} be semifinite.
Denote $P_{fin}(\cM):=P(\cM)\cap \mathcal{F}(\tau)$.
For each $e\in P(\cM)\cap E(\cM,\tau)$ and $x\in \cF(\tau)$,
let $u_e,J(e),J(x)$ be as in the proof of  Theorem \ref{positive}. 
For any $e\in P_{fin}(\cM)$, we denote $ J_e(x):=J(exe),  x\in \cM$. 
The collection $\{J_e:e\in P_{fin}(\cM)\}$ forms a family of normal Jordan $^*$-homomorphisms.
By \cite[Lemma 2.18]{HSZ20}, 
there exists a normal Jordan $^*$-homomorphism (still denoted by $J$)
$J:\cM \to \cN$ agreeing with $J_e$ 
for every $e\in P_{fin}(\cM)$.

\begin{theorem}\label{infiniteTform}
   Suppose that 
   $\cM$ and $\cN$ are semifinite von Neumann algebras equipped with semifinite faithful normal traces $\tau$ and $\nu$, respectively.
    Let $E(\mathcal{M}, \tau)$ and $F(\cN,\nu)$ be 
 an  $\cM$- and an $\cN$-bimodule, respectively.
    If
     $T: E(\cM, \tau)\xrightarrow{\rm into} F(\cN, \nu)$ 
     is o-$t_{lm}$ continuous and disjointness-preserving, 
     then $T$ has the form
     \[
     T(x)=wbJ(x), \quad  x\in E(\cM,\tau)\cap \cM,
    \]
     where $w$ is a partial isometry in $\cN$,
     $b$ is a (possibly not measurable) positive self-adjoint operator affiliated with $\cN$ and
     $J:\cM\to \cN$ is a normal Jordan $^*$-homomorphism.
     {\color{blue}Moreover, $e^b(\delta)\in Z(J(\cM))$ for all $\delta\subset\mathbb{R}$.}
\end{theorem}

\begin{proof}
Arguing similarly as the proof of Theorem \ref{positive}, we obtain that 
if $p,q\in P_{fin}(\cM)$ satisfying $pq=0$, 
then
\begin{equation}\label{upjq}
r(u_p)=
 r(T(p))=r (|T(p)|) 
\stackrel{\mbox{\tiny\cite[Eq.(24)]{HSZ20}}}{=}J(p),
\quad u_pJ(q)=u_qJ(p)=0,
\end{equation}
and for any  $\alpha, \beta\in\mathbb{R}$, 
\begin{align}\label{linear2}
    T(\alpha p+\beta q)=u_{p+q}(|\alpha T(p)|+|\beta T(q)|),
    \quad |T(\alpha p+\beta q)|=|\alpha T(p)|+|\beta T(q)|. 
\end{align} 
For each $e,f \in P_{fin}(\cM) $ such that $e\le f$,
it follows from \eqref{upjq} and \eqref{linear2} that 
\begin{equation}\label{upJe=ueJp}
\begin{aligned}
     u_f J(e ) =u_e J(e)=u_e . 
\end{aligned}
\end{equation}

Let $\{p_i\}_{i\in I} $ be the increasing net of all projections in $P_{fin}(\cM)$. 
Note  that   $r(u_{p_i})=J(p_i)\uparrow J(\mathbf{1})$  ($J$ is normal) 
and $u_{p_i}J(p_j)=u_{p_j}$ for $i,j\in I$ with $i\ge j$ (see \eqref{upJe=ueJp}).
By Lemma \ref{solim}, we may define 
\begin{align}\label{u1}
    u_\mathbf{1}:=so-\lim_i u_{p_i}\in \cN.
\end{align}
In particular, $u_\mathbf{1}=u_\mathbf{1}J(\mathbf{1})$.
By \cite[Proposition 1.3.2]{DPS}, we have $J(\mathbf{1})=so-\lim_iJ(p_i)$. 
For any $e\in P_{fin}(\cM)$, since $J(e)$ is bounded, it follows that 
\begin{align}\label{u1Je=ue}
    u_\mathbf{1}J(e)
\stackrel{\eqref{u1}}{=}so-\lim_{i}u_{p_i}J(e)
\stackrel{\eqref{upJe=ueJp}}{=} u_e.
\end{align}
Hence,
\[
T(e)=u_e|T(e)|
\stackrel{\eqref{u1Je=ue}}{=}u_\mathbf{1}J(e)|T(e)|\stackrel{\eqref{upjq}}{=} u_\mathbf{1}|T(e)|.
\]
Moreover, it follows from \eqref{u1Je=ue} that $J(p_i)u_\mathbf{1}^*=u_{p_i}^*$,
which implies that
\begin{align*}
u_\mathbf{1}u_\mathbf{1}^*&
=so-\lim_i u_{p_i}u_\mathbf{1}^*
=so-\lim_i u_{p_i}J(p_i)u_\mathbf{1}^*\\&
=so-\lim_i u_{p_i}u_{p_i}^*
=so-\lim_i l(T(p_i))\in P(\cN).
\end{align*}
Hence,  $u_\mathbf{1}$ is a partial isometry (see \cite[p. 18]{DPS}).

For each $0\le x\in E(\cM,\tau)$, there exists an upward directed system 
$\{x_\lambda\}\subset\cF(\tau)$ of positive simple operators such that 
$0\le x_\lambda\uparrow x$ in $S_h(\cM,\tau)$ \cite[Proposition 2.3.12]{DPS}.
Arguing as the proof for \eqref{Tlmcon} above, we obtain that 
$u_\mathbf{1}^*T(\cdot)$ is a normal (in particular, positive) disjointness-preserving mapping.
Let $w:=u_\mathbf{1}$.
Arguing  mutatis mutandis as in the  proof of \cite[Theorem 3.6]{HSZ20}, we have
$$u_{\bf 1}^* T(x) = bJ(x), \quad  x\in E(\cM,\tau)\cap \cM,$$
where $b$ is a limit of $\{|T(e)|\}_{e\in P_{fin}(\cM)}$ in the strong resolvent sense. In particular, $b$ is affiliated with $\cN$ (see \cite[Proposition 3.8]{HSZ20}).
Therefore, 
\[
T(x)=wbJ(x), \quad  x\in E(\cM,\tau)\cap \cM,
\]
which completes the proof. 
\end{proof}

\section{A noncommutative Abramovich's theorem}

The main purpose of this section is to establish, by employing two distinct approaches,
a noncommutative version of Abramovich's result\cite[Theorem 1]{Abra1991}, see Theorem \ref{main1} above.
The first approach, detailed in  Section \ref{reducecommute}, is based on the construction of a $^*$-isomorphism from a commutative von Neumann subalgebra to the function space, which allows for a direct application of Abramovich's result\cite[Theorem 1]{Abra1991}. The second approach, detailed in  Section \ref{analoproof},
provides a noncommutative analogue of the proof of \cite[Theorem 1]{Abra1991}.

\subsection{Proof by reduction to the commutative case}\label{reducecommute}

To establish our results, we need the following lemmas.


\begin{lemma}\label{isom}
    Let $\phi:\cA\to\cB$ be a  normal $^*$-homomorphism between von Neumann algebras $\cA$ and $\cB$. Assume $\cA$ is abelian. Then there exists a projection $p\in P(\cA)$ such that
    \begin{enumerate}
        \item $\phi|_{\cA p}$ is a $^*$-isomorphism from $\cA p$ onto $\phi(\cA)$;
        \item $\phi(\cA(\mathbf{1}-p))=0$.
    \end{enumerate}
\end{lemma}

\begin{proof}
    Denote
    \[
    P(\cA)_\phi:=\{ p\in P(\cA):\phi(p)= 0\}.
    \]
    For an arbitrary increasing net $\{q_i\}_{i\in I}\subset P(\cA)_\phi$ satisfying $q_i\uparrow q\in P(\cA)$, it follows from the normality of $\phi$ that $q\in P(\cA)_\phi$.
    Noting that  $P(\cA)_\phi$ is not empty, by Zorn's Lemma, there exists a maximal element $e$ of $P(\cA)_\phi$.
    Let $p:=\mathbf{1}-e$. 
    Then, $\phi(\cA(\mathbf{1}-p))=\phi(\cA e)=\phi(\cA)\phi(e)=0$.
    
    We claim that $\phi|_{\cA p}$ is a $^*$-isomorphism. 
    Indeed, if there exists $0\ne x\in \cA p$ satisfying $\phi(x)=0$, then there exists a constant $\lambda>0$ such that $e^{|x|}(\lambda,\infty)\ne 0$. Since $\phi$ is positive and $\lambda e^{|x|}(\lambda,\infty)\le |x|$ in $(\cA p)_+$, it follows that 
    \[
    0\le\phi(e^{|x|}(\lambda,\infty))\le\frac{1}{\lambda}\phi(|x|)=0,
    \]
    i.e., $\phi(e^{|x|}(\lambda,\infty))=0$.
    Hence, $e^{|x|}(\lambda,\infty)\le r(x)\le p$ and $e^{|x|}(\lambda,\infty)\in P(\cA)_\phi$, which implies that $e^{|x|}(\lambda,\infty)\vee e\ge e$ and $e^{|x|}(\lambda,\infty)\vee e\ne e$.
    This contradicts the maximality of $e$.
    Therefore, $\phi|_{\cA p}$ is injective.
 Recall that $\phi(\cA(\mathbf{1}-p))=0$, which implies $\phi(\cA p)=\phi(\cA)$.
    Then we have $\phi|_{\cA p}$ is a $^*$-isomorphism from $\cA p$ onto $\phi(\cA)$. This proves our claim and completes the proof.

\end{proof}

For convenience of the reader, we recall the definition of  reduced von Neumann algebras. For each  $e\in P(\cM)$, define
\[
\cM_e:=e\cM e=\left\{exe:x\in \cM\right \},
\]
which is a $^*$-subalgebra of $\cM$ with unit element $e$.
It is clear that $P(\cM_e)=\{p\in P(\cM): p\leq e\}$.
Define $\tau_e:(\cM_e)_+\to [0,\infty]$ by 
\[
\tau_e(x_e)=\tau(exe), \quad  \,x\in \cM_+,
\]
that is,
$\tau_e=\tau|_{e\cM e}$. 
The $\tau_e$-measurable operator space $S(\cM_e,\tau_e)$ is unital $^*$-isomorphic to $eS(\cM,\tau)e$ (see \cite[Lemma 3.7.1]{DPS}). For each $x_e=exe\in \cM_e$, its singular function satisfies
\[
\mu(x_e)=\mu(exe).
\]
Hence, $\tau_e$ can be extended to $S(\cM_e,\tau_e)_+$, satisfying
\[
\tau_e(x_e)=\tau(exe), \quad  x\in S(\cM,\tau)_+.
\]

The following lemma is an extension of \cite[Lemma 1.3]{CKS}. 
\begin{lemma}\label{AandB}
Let \(\cM\) be an atomless von Neumann algebra with a faithful normal tracial state \(\tau\). Suppose \(\cA\subset\cM\) is an abelian von Neumann subalgebra isomorphic to \(L_\infty(0,1)\). 
If \(a\in S(\cM,\tau)\) is self-adjoint and commutes with $\cA$, then there exists an abelian von Neumann subalgebra 
\(\cB\subset\cM\) with \(\cA\cup\{a\}\subset S(\cB, \tau|_{\cB})\) and 
$\cB$ is trace-preserving $^*$-isomorphic to $L_\infty(0,1)$ equipped with the Lebesgue measure.

In particular, for any nonzero self-adjoint $a\in S(\cM,\tau)$,  there exists an abelian von Neumann subalgebra 
\(\cB\subset\cM\) with \(a\in  S(\cB, \tau|_{\cB})\) and 
$\cB$ is trace-preserving $^*$-isomorphic to $L_\infty(0,1)$ equipped with the Lebesgue measure.
\end{lemma}

\begin{proof} Let \(\cC:=W^*(\cA\cup S)\) be the von Neumann subalgebra of \(\cM\) generated by \(\cA\) and 
\(S=\left\{e^{a}\left(-\frac{l}{n},\frac{k}{n}\right): n,k,l\in \mathbb{N}\right\}\). Since \(a\) commutes with \(\cA\) and \(\cA\) is abelian, \(\cC\) is also abelian.
The algebra $\cC$ is decomposed into direct sum of atomless and atomic parts. 
Moreover, since \(\cM\) is a tracial von Neumann algebra, there exists
at most countable system $\{p_0\}\cup\{p_n\}_{n\in I}$ of mutually orthogonal projections in \(\cC\) such that
\[
\cC = \Big(\bigoplus_{n\in I} \bC p_n\Big)\oplus \cC p_0,
\]
where \(\cC p_0\) is atomless.  Taking into account that $\cA\equiv L_\infty(0,1)$ has the separable predual and $S$ is countable, we obtain that $\cC p_0\equiv L_\infty(0,1)$ (see the proof of \cite[Theorem 3.5.2]{SS08}).
Since $\cM$ is atomless, for each $n \in I$, we can find\footnote{
Indeed, let $p\in \cM$ be a nonzero projection. Since $\cM_p$  is atomless, there exists (see \cite[p. 276]{SS08}) a family of
projections $\left\{e^{(n)}_j : 1 \le  j \le  2^n, n \ge 0\right\}$ in $\cM_ p$ such that
\begin{enumerate}
\item[(1)] $e^{(0)}=p$,
\item[(2)] $e^{(n+1)}_{2j-1}+e^{(n+1)}_{2j} = e^{(n)}_j$,
\item[(3)] $\tau\left(e^{(n)}_j\right) = 2^{-n}\tau(p)$,
\end{enumerate}
for $1 \le  j \le  2^n$ and $n \ge 0$. Let
\[
\cB_p = \left\{e^{(n)}_j : 1 \le  j \le  2^n, n \ge 0\right\}''\cap p\cM p
\]
be the abelian von Neumann subalgebra of $\cM_p$ generated by these elements. 
}  
inside \(\cM_{p_n}\) an atomless  abelian von Neumann subalgebra 
\(\cB_n\subset \cM_{p_n}\) with unit \(p_n\)
which is generated by a countable family of projections.
Setting
\[
\cB = \Big(\bigoplus_{n\in I} \cB_n\Big)\oplus \cC p_0,
\]
we obtain the required subalgebra. 
By \cite[Theorem 3.5.2]{SS08},
we have 
$\cB$ is trace-preserving $^*$-isomorphic to $L_\infty(0,1)$ equipped with the Lebesgue measure,
which completes the proof of the first part.

In order to prove the second statement, we only need to take $\cC=W^*(S)$.
The proof is complete.
\end{proof}

Given $0<a\leq \infty$, for any measurable function $f$ on $(0,a)$ and $s>0$, the dilation $D_s$ of $f$ is defined by setting\cite[p. 392]{DPS}
\begin{equation*}
    D_sf(t)=\begin{cases}
        f(\frac{t}{s}), &\text{if}\; t\in(0,a)\; \text{and}\; \frac{t}{s}\in(0,a).\\
        0, &\text{if}\; t\in(0,a)\; \text{and}\; \frac{t}{s}\not\in(0,a).
    \end{cases}
\end{equation*}
If $E$ is a Calkin function space and $f\in E$, we have (the case for symmetric function spaces can be found in \cite[p. 96 and Corollary 2 in p. 98]{KPS})
\begin{align}\label{Ds}
    D_sf\in E.
\end{align}

\begin{proof}[Proof of Theorem \ref{main1}]

Without loss of generality, we may assume that $\tau$ and $\nu$ are tracial states, i.e., $\tau(\mathbf{1})=\nu(\mathbf{1})=1$.

If $T$ is o-$t_m$ continuous, then
it follows from the proof of Theorem $\ref{positive}$ that $w^*T(\cdot)$ is a normal mapping from $E(\cM,\tau)$ into $F(\cN,\nu)$.
Hence, without loss of generality, we may assume that $T$ is normal (in particular, it is positive). 

Suppose that $T\not\equiv0$.
By Theorem \ref{normal},
there exists a normal Jordan $^*$-homomorphism $J:S(\cM,\tau)\to S(\cN,\nu)$ satisfying
\begin{align}\label{T(x)}
    T(x)=T(\mathbf{1})J(x),\quad  x\in E(\cM,\tau).
\end{align}
Moreover, $T(\mathbf{1})$ commutes with $J(x)$ for every $x\in E(\cM,\tau)$.

Let $0\le x_0\in E(\cM,\tau)$ such that $T(x_0)\ne0$.
By \cite[Lemma 1.3]{CKS} (or Lemma~\ref{AandB}), there exists an atomless abelian von Neumann subalgebra $\cA$ of $r(x_0)\cM r(x_0)$ containing all spectral projections of $x_0$.
Moreover, there exists a trace-preserving $^*$-isomorphism 
$$
\Phi:S(\cA ,\tau|_\cA)\to S\left(0,\tau\left(r(x_0)\right)\right),
$$
where the trace on $S\left(0,\tau\left(r(x_0)\right)\right)$ is the Lebesgue measure. In particular, $\mu(x)=\mu(\Phi(x))$ for each $x\in S(\cA ,\tau|_\cA)$.

Since $\cA$ is abelian, it follows that $J(\cA)$ is an abelian von Neumann algebra (see Proposition \ref{Jcom} and Remark \ref{Jsubalg}).
By Lemma \ref{isom}, there exists a projection $p\in P(\cA)$ such that
$J|_{\cA p}$ is a $^*$-isomorphism and $J(\mathbf{1}_{\cA} -p)=0$.
It can be easily checked that $\nu$ is a finite faithful normal trace on $J(\cA)$ and $J(p)$ is the unit of $J(\cA)$.
Noting that $J$ is continuous in the measure topology 
(see Proposition \ref{Jnormea}), we have
\begin{align}\label{Jisomorphism}
    J:pS(\cA,\tau|_\cA)p\to S(J(\cA),\nu|_{J(\cA)}) 
\end{align}
is a $^*$-isomorphism (see  Section \ref{s:Jordan}).

For each $x\in E(\cM,\tau)\cap S(\cA,\tau|_\cA)$, we have
\begin{align}\label{T|_A}
    T(x)\stackrel{\eqref{T(x)}}{=}T(\mathbf{1})J(x)
    \stackrel{J(x)=J(px)}{=}
    T(\mathbf{1})J(p)J(x)
    \stackrel{\eqref{T(x)}}{=}T(p)J(x).
\end{align}
Moreover, since $T(\mathbf{1})$ commutes with $J(x)$, 
arguing similarly to \eqref{T|_A}, we obtain that  $T(p)$ commutes with $J(x)$, which implies that each spectral projection of $T(p)$ commutes with $J(\cA)$ (see \cite[Proposition 2.2.22]{DPS}).

There exists a measurable subset $D$ of $(0,\tau(r(x_0)))$ satisfying $\Phi(p)=\chi_D$.
Then, we have $\Phi|_{\cA_p}$ is a $^*$-isomorphism from $\cA_p$ to $L_\infty(D)$, which implies that $J(\cA)$ is $^*$-isomorphic to $L_\infty(D)$.
Applying Lemma~\ref{AandB} to $J(\cA)$ and $T(p)$ we can construct an atomless abelian von Neumann algebra $\cB$ containing\footnote{Indeed, we may assume that  $\cB$ is  the von Neumann algebra generated by $J(\cA)$ and  spectral projections of $T(p)$. 
Since 
$J(\cA)$ is atomless, it follows  that $\cB$ is atomless. 
} $J(\cA)$ and all spectral projections of $T(p)$.
Moreover, there exists a trace-preserving $^*$-isomorphism 
$\Psi$ from $\cB$ onto $L_\infty(0,\nu(J(p)))$ equipped with the Lebesgue measure.
It follows from \cite[Proposition 2.9.3]{DPS} that $\Psi$ can be uniquely extended to a trace-preserving $^*$-isomorphism (still denoted by $\Psi$)
$$
\Psi:S(\cB,\nu|_\cB)\to S(0,\nu(J(p))).
$$
By \cite[Proposition 2.9.2(iv)]{DPS},   $\Psi$ preserves the singular value functions.

Let $s:=\frac{1}{\tau\left(r(x_0)\right)}\ge 1$.
Noting that 
\[
D_s:S\left(0,\tau\left(r(x_0)\right)\right)\xrightarrow{\rm onto} S(0,1)
\]
is a $^*$-isomorphism.
It is readily verified that $D_s^{-1}, \Phi^{-1}, \Psi$ are disjointness-preserving and normal (see  Section \ref{s:Jordan}). 

Since $s\ge 1$, it follows that
$\mu(D_s^{-1}(f))=D_s^{-1}(\mu(f))\le \mu(f)$ for each $ f\in E(0,1)$ 
(see \cite[p. 96]{KPS}), which implies that $\mu(D_s^{-1}(f))\in E(0,1)$.
Moreover, the fact that $\Phi$ preserves singular value function (see above) yields that
\begin{align}\label{comb}
    \left( \Phi^{-1} \circ D_s^{-1}\right)(f)\in E(\cM,\tau)\cap S(\cA,\tau|_\cA),\quad f\in E(0,1).
\end{align}

Let $x\in E(\cM,\tau)\cap S(\cA,\tau|_\cA)$ be arbitrary.
Then, we have
\[
J(x)\stackrel{J(\mathbf{1}_{\cA} -p)=0}{=}
J(pxp)\stackrel{\eqref{Jisomorphism}}{\in}S(J(\cA),\nu|_{J(\cA)})\subset S(\cB,\nu|_\cB)
\]
(see \cite[p. 113]{DPS}).
Noting that $T(p)\in S(\cB,\nu|_\cB)$ (see \cite[Proposition 2.1.4(iv)]{DPS}),
it follows from \eqref{T|_A} that $T(x)\in F(\cN,\nu)\cap S(\cB,\nu|_\cB)$.
Recall that $\Psi$ is trace-preserving, which together with \eqref{comb} yields that
\[
T_1(f):=\left(\Psi\circ T\circ \Phi^{-1} \circ D_s^{-1}\right)(f),
\quad  f\in E(0,1)
\]
is a normal disjointness-preserving mapping from $E(0,1)$ into $F(0,1)$.

Since $\Phi$ is trace-preserving and  $x_0\in E(\cM,\tau)\cap S(\cA,\tau|_\cA)$,
it follows from \eqref{Ds} that 
\[
\left(D_s\circ\Phi\right)(x_0)\in E(0,1),
\]
which together with $T(x_0)\ne 0$ yields that
$T_1\not \equiv 0$.
Moreover, by \cite[Lemma 1.24]{AA}, $T_1$ is normal implies that $T_1$ is order continuous.
Then by \cite[Theorem 1]{Abra1991}, we have $E(0,1)\subseteq F(0,1)$, which yields a contradiction.
Therefore, $T\equiv0$.
\end{proof}

\subsection{A noncommutative version of Abramovich's Proof}\label{analoproof}

The following lemma is   a noncommutative  version of \cite[Lemma 5]{Abra1991}, whose proof relies on the noncommutative Radon--Nikodym theorem \cite[Theorem 3.4.24]{DPS}.


\begin{lemma}\label{traceineq}
	Suppose that
	$\cM$ and $\cN$ are  
semifinite 
	  von Neumann algebras equipped with semifinite faithful normal traces $\tau$ and $\nu$, respectively. 
      Let $J: \cM  \xrightarrow{\rm into} \cN$ be a normal Jordan $^*$-homomorphism.
      If $0\neq p\in P(\cM)$ with 
	$0<\nu(J(p) )<\infty$,
	then  there exists a  nonzero projection $  q\le   p $ and a constant $C>0$  such that 
	for arbitrary projection 
	$e\leq q$, 
	we have
	\begin{align*}
		C^{-1}\tau(e) \leq	\nu(J (e))\leq C\tau (e).
	\end{align*}
\end{lemma}

\begin{proof}
 
 Let $\cM_p=p\cM p$. Then $\nu(J(\cdot))$ is a finite faithful normal  trace on $\cM_p$.
By the noncommutative Radon--Nikodym theorem \cite[Theorem 3.4.24]{DPS}, there exists a unique nonzero positive element $a\in L_1 (\cM_p, \tau_p)$ such that
\begin{align}\label{Radon}
    \nu\left(J(x)\right)=\tau(ax), 
    \quad x\in \cM_p.
\end{align}

Consider the spectral projection $q:=e^a{\left[C^{-1}, C\right]}$, 
where $C>0$ is a constant such that 
	$0\ne  e^a{\left[C^{-1}, C\right] } \in \cM_p $.  
	For  arbitrary nonzero projection 
	$e\leq q$, 
	we have 
	\begin{align}\label{tau12}
	    \nu\left(J(e)\right)\stackrel{\eqref{Radon}}{=}
        \tau (ae)=\tau(q a q e ) =\tau(eqaqe).
	\end{align}
  Noting that 
   $C^{-1}q\le qaq \le Cq$, 
   it follows from \cite[Proposition 2.2.24(iv)]{DPS} that
    \[
    e\left(C^{-1}q\right)e\le eqaqe \le e\left(Cq\right)e,
    \]
    which implies that
    $$
   C^{-1} \tau(e) =   \tau \left(e\left(C^{-1}q\right)e\right)  \le \tau(eqaqe) \stackrel{\eqref{tau12} }{=} 
   \nu\left(J(e)\right)  
   \le\tau(e\left(Cq\right)e)=C\tau(e).
    $$
The proof is completed.
\end{proof}

The following two lemmas   are obvious and therefore we omit their proofs. 
 For the case of  rearrangement invariant ideals on  a finite nonatomic measure space, the corresponding results  can be found in  \cite[Lemmas 6 and 7]{Abra1991}.

\begin{lemma}\label{choosestepfunction}
    Suppose that $(\cM,\tau)$ is an atomless von Neumann algebra with 
    a finite faithful normal  trace $\tau$.  
    Let $E(\mathcal{M}, \tau)$ and $F(\mathcal{M}, \tau)$  be two Calkin operator spaces.
    If $E(\mathcal{M}, \tau)\not\subseteq F(\mathcal{M}, \tau)$, then for each $0\neq p\in P(\cM)$, 
    there exists a simple operator
	$x= \sum\limits_{k\geq 1} t_k p_k\in E(\mathcal{M}, \tau)\backslash F(\mathcal{M}, \tau)$, where $0<t_k\in\mathbb{R}$ and $\{p_k\}_{k\geq 1}$ is a sequence of pairwise disjoint projections in $\cM$, such that $r(x)=\sum\limits_{k\ge 1}p_k\le p$.
\end{lemma}

\begin{lemma}\label{deduce}
	Suppose that $\cM$ and $\cN$ are atomless von Neumann algebras equipped with   faithful normal   tracial 
    states $\tau$ and $\nu$, respectively.
    Let $E(\mathcal{M}, \tau)$ and $E(\cN,\nu)$ be the Calkin operator spaces corresponding to $E(0,1)$.
    Assume that  $x= \sum\limits_{k\geq 1} t_k p_k \in E(\mathcal{M}, \tau)$, where $0<t_k\in\mathbb{R}$ and $\{p_k\}_{k\geq 1}$ is a sequence of pairwise orthogonal projections in $\cM$.
	 If $\{q_k\}_{k\geq 1}$ is a sequence of pairwise orthogonal projections in $\cN$ satisfying 
	$$
    \frac{\nu(q_k)}{\tau(p_k)} \leq C
    $$
	for some $C> 0$, then
	$y= \sum\limits_{k\geq 1} t_k q_k\in E(\cN, \nu)$.
\end{lemma}

\begin{proof}[Proof of Theorem \ref{main1}]\label{second}

Without loss of generality, we may assume that $\tau$ and $\nu$ are tracial states, i.e., $\tau(\mathbf{1})=\nu(\mathbf{1})=1$.

Suppose that $T\not\equiv0$.
As in the proof of Theorem \ref{main1} in subsection \ref{reducecommute}, we may assume that $T$ is normal (in particular, positive).
It follows from Theorem \ref{normal} that there exists a normal Jordan $^*$-homomorphism $J:S(\cM,\tau)\to S(\cN,\nu)$ satisfying
\begin{align}\label{Tx}
    T(x)=T(\mathbf{1})J(x),\quad  x\in E(\cM,\tau).
\end{align}
Moreover, $T(\mathbf{1})$ commutes with $J(x)$ for each $x\in E(\cM,\tau)$.

Since $T\not\equiv0$, it follows that there exists a constant $\varepsilon>0$ such that the spectral projection 
$q_\varepsilon:=e^{T(\mathbf{1})}(\varepsilon,\infty)\ne 0$.
Define
\[
T_\varepsilon(x):=T(x)q_\varepsilon,\quad   x\in E(\cM,\tau).
\]
For each $x\in E(\cM,\tau)$, it follows from $T(\mathbf{1})$ commutes with $J(x)$ that $T(\mathbf{1})$ commutes with $T(x)$ (see \eqref{Tx}).
By \cite[Proposition 2.2.22]{DPS}, $q_\varepsilon$ commutes with $T(x)$,
which implies that $T_\varepsilon$ is disjointness-preserving.
Moreover, since $T$ is normal, it follows that
\[
T_\varepsilon(x)=q_\varepsilon T(x)q_\varepsilon
\]
is also normal (see \cite[Proposition 2.2.25(iii)]{DPS}).
Consequently, 
$$T_\varepsilon: E(\cM,\tau)\xrightarrow{\rm into} F(\cN,\nu)$$ 
is a nonzero normal disjointness-preserving mapping.
Without loss of generality, we may replace $T$ and $J$
with $T_\varepsilon$ and $J_\varepsilon$, respectively (still denoted by $T$ and $J$).
Furthermore, it follows from $r(T(\mathbf{1}))=J(\mathbf{1})$ 
(see \eqref{upjq}) that
\begin{align}\label{epsilonres}
    T(\mathbf{1})\ge \varepsilon J(\mathbf{1}).
\end{align}

It follows from Lemma \ref{traceineq} that, there exists a nonzero projection $q\in\cM$ and a constant $C>0$ such that for arbitrary projection $e\le q$,
we have 
\begin{align}\label{ineq1}
    C^{-1}\tau(e)\le \nu(J(e))\le C\tau(e).
\end{align}
Let $F(\cM,\tau)$ be the Calkin operator space corresponding to $F(0,1)$.
By Lemma~\ref{choosestepfunction}, there exists a positive simple operator
\begin{align}\label{xform}
x=\sum\limits_{k\ge 1}t_ke_k\in E(\cM,\tau)\backslash F(\cM,\tau),
\end{align}
where $0<t_k\in\mathbb{R}$ for all $k\ge 1$ and $\{e_k\}_{k\geq 1}$ is a sequence of pairwise orthogonal projections in $\cM$.
Define 
\begin{align}\label{yform}
    y:=\sum\limits_{k\ge 1}t_kJ(e_k)\in S(\cN,\tau).
\end{align}
Observe that $\{J(e_k)\}_{k\geq 1}$ is a sequence of pairwise orthogonal projections in $\cN$ (see Proposition \ref{Jcom}).
It follows from the normality of $J$ that 
\begin{align*}
    J(x)\stackrel{\eqref{xform}}{=}
    J\left(\sum\limits_{k\ge 1}t_ke_k\right)=\sum\limits_{k\ge 1}t_kJ(e_k)
    \stackrel{\eqref{yform}}{=}y.
\end{align*}
Then, we have
\[
T(\mathbf{1})y=T(\mathbf{1})J(x)
\stackrel{\eqref{Tx}}{=}T(x)\in F(\cN,\nu).
\]
Moreover, noting that $T(\mathbf{1})$ commutes with $y$, 
by \cite[Proposition 2.2.22 and Proposition 2.2.24(iv)]{DPS},
we have
\[
T(\mathbf{1})y=y^{1/2}T(\mathbf{1})y^{1/2}\stackrel{\eqref{epsilonres}}{\ge}
y^{1/2}\varepsilon J(\mathbf{1})y^{1/2}=\varepsilon y,
\]
which implies that $\varepsilon y\in F(\cN,\nu)$, i.e., $y\in F(\cN,\nu)$.

On the other hand, since $e_k\le q$ for every $k$, it follows from \eqref{ineq1} that
\[
C^{-1}\tau(e_k)\le \nu(J(e_k))\le C\tau(e_k)
\]
for all $k\ge 1$. 
Recall that $y\in F(\cN,\nu)$. By Lemma \ref{deduce}, 
we have $x\in F(\cM,\tau)$,
which contradicts \eqref{xform}.
Therefore, $T\equiv0$. The proof is complete.	
\end{proof}

\subsection{Applications}

If the Calkin operator  space $E(\cM,\tau)$ is equipped  with an order continuous symmetric  norm, then we have the following corollary.

\begin{corollary}\label{ordercontinuournorm}
Assume that $\cM$ and $\cN$ are  two  atomless von Neumann algebras equipped with   faithful normal tracial states $\tau$ and  $\nu$ respectively. 
  Let   
    $E(\cM,\tau)$ (respectively,  $F(\cN,\nu)$) be the  symmetrically normed operator space  corresponding to the  symmetrically normed function space $E(0,1)$ (respectively, $F(0,1)$).
If $E(\cM,\tau)$ has order continuous norm and  $T: E(\mathcal{M}, \tau)\to F(\cN, \nu)$ is a nonzero disjointness-preserving bounded mapping, then 
	\[
	   E(0,1) \subseteq F(0,1).
	\]
Moreover, there exists a Jordan $^*$-homomorphism $J:\cM\to \cN$ such that
\[
E(J(\cM),\nu)\subseteq F(\cN,\nu),
\]
where $E(J(\cM),\nu)$ is defined by
\begin{align*}
    E(J(\cM),\nu):=\left\{x\in S(J(\cM),\nu): \mu(x)\in E(0,1)\right\}.
\end{align*}
\end{corollary}

\begin{proof}

Suppose that $\{x_i\}_{i\in I}\subset E(\cM,\tau)$ satisfying $x_i\xrightarrow{(o)}0$.
Then, there exists a net $\{y_i\}_{i\in I}\subset E(\cM,\tau)_+$ satisfying $|x_i|\leq y_i\downarrow 0$.
Since $\norm{\cdot}_E$ is order continuous, it follows that
\[
\norm{x_i}_E\le \norm{y_i}_E\to 0
\]
Consequently,
\[
\norm{T(x_i)}_F\le \norm{T}\norm{x_i}_E\to 0
\]
By \cite[Proposition 4.4.4]{DPS}, we have $T(x_i)\xrightarrow{ t_m} 0$, which implies that $T$ is o-$t_{m}$ continuous.
Then, it follows from Theorem \ref{main1} that
\begin{align}\label{EinF}
    E(0,1) \subseteq F(0,1).
\end{align}

It follows from Proposition \ref{ocnorm} that 
\begin{align*}
    T(x)=w|T(\mathbf{1})|J(x),\quad  x\in E(\cM,\tau),
\end{align*}
where $J:S(\cM,\tau)\to S(\cN,\nu)$ is a normal Jordan $^*$-homomorphism.
Then $J(\cM)$ is a von Neumann algebra (see Remark \ref{Jsubalg}) and $\nu$ on $J(\cM)$ is naturally a finite faithful normal trace. 
By \eqref{EinF}, the proof is complete. 
\end{proof}

The following result is a direct consequence of Theorem \ref{main1}.

\begin{corollary}
Assume that $\cM$ and $\cN$ are  two  atomless von Neumann algebras equipped with  faithful normal tracial states $\tau$ and  $\nu$, respectively.
There is no nonzero bounded disjointness-preserving mapping
from $L_p( \cM, \tau)$
into $L_q( \cN, \nu)$, 
where $0<p<q\leq \infty$.
\end{corollary}


A linear injection $T$ between two Calkin spaces $E$ and $F$ is an order monomorphism whenever $T(x)\ge 0$ if and only if $x\ge 0$ for each $x\in E$.
If $T$ is additionally bijection, then $T$ is an order isomorphism.

\begin{remark}
    An order isomorphism between Banach lattices (in particular, Calkin function spaces) is a lattice isomorphism (see, e.g., \cite[Theorem 2.15]{AB1985}) and hence is disjointness-preserving.
    Consequently, two order isomorphic Calkin function spaces coincide (see, e.g., \cite[Corollary 4]{Abra1991}).
    
    However, in the noncommutative setting, order isomorphism is not necessary disjointness-preserving. 
   Indeed, suppose that 
    $\cM=L_\infty(0,1)\bar{\otimes} M_2(\mathbb{C})$.
    It follows from \cite[Lemma 3.7.5]{DPS} that $\cM$ is atomless.
    Define
    \begin{align*}
        a:=\chi_{(0,1)}\otimes
        \begin{pmatrix}
            1 & 1 \\
            0 & 1
        \end{pmatrix}
    \end{align*}
    and $T:\cM\to \cM$ as follows
    \[
    T(x)=axa^*,\quad  x\in \cM.
    \] 
    It follows from \cite[Proposition 2.2.24(iv)]{DPS} that $T$ and $T^{-1}$ are both positive.
    Let
    \begin{align*}
        x=\chi_{\left(0,1\right)}\otimes
        \begin{pmatrix}
            1 & 0 \\
            0 & 0
        \end{pmatrix},\quad
        y=\chi_{\left(0,1\right)}\otimes
        \begin{pmatrix}
            0 & 0 \\
            0 & 1
        \end{pmatrix}.
    \end{align*}
    Then,  $xy=0$. 
    However, both $T(x), T(y)$ are positive and 
    \[
    T(x)T(y)=\chi_{\left(0,1\right)}\otimes
    \begin{pmatrix}
        1 & 1\\
        0 & 0
    \end{pmatrix}\neq 0,
    \]
    which implies that $T$ is not disjointness-preserving.
\end{remark}

Now we consider the case when the von Neumann algebras and the traces are 
semifinite (infinite). 
Theorem \ref{main1} dose not hold in this case (see \cite[Section 5.4]{Abra1991}).
Nevertheless, for operators with  
$\tau$-finite right support projections (equivalently, left support projections), 
we obtain a result consistent with 
Theorem~\ref{main1}.

\begin{theorem}\label{infinite}
    Suppose that $\cM$ and $\cN$ are semifinite atomless von Neumann algebras equipped with semifinite faithful normal traces $\tau$ and $\nu$, respectively.
    Let $E(\cM,\tau)$ (respectively, $F(\cN,\nu)$) be a Calkin operator space corresponding to $E(0,\infty)$ (respectively, $F(0,\infty)$).
    If there exists a nonzero o-$t_{lm}$ continuous  (or normal) and disjointness-preserving mapping $T: E(\cM, \tau)\xrightarrow{\rm into} F(\cN, \nu) $, then for each $x_0\in E(\cM, \tau)$ with $\tau$-finite $r(x_0)$, we have $\mu(x_0)\in F(0,\infty)$.
\end{theorem}

To establish this theorem, we first introduce the following lemma which serves as an essential tool.

\begin{lemma}\label{notbounded}
Assume there exists a nonzero o-$t_{lm}$ continuous  (or normal) and disjointness-preserving mapping $T: E(\cM, \tau)\xrightarrow{\rm into} F(\cN, \nu)$.
If $E(\cM,\tau)\not\subset  \cM$, then 
      there exists a positive  element $x \in E(\cM,\tau)$ with $\tau$-finite support such that $T(x)\not\in \cN$. 
    
\end{lemma}

\begin{proof}

By the proof of Theorem \ref{infiniteTform}, we have $w^*T(\cdot)$ is normal. Without loss of generality, we may assume that $T$ is normal (in particular, positive).

Since $T\not \equiv 0$, it follows that there exists a constant $K>0$
and a projection $p_0\in E(\cM,\tau)$ satisfying $e^{T(p_0)}(K,\infty)\ne 0$.
Noting that $\tau$ is semifinite and $T$ is normal, we may assume that $\tau(p_0)<\infty$.
Denote
\[
P_0:=\left\{ p\le p_0:\norm{T(p)}_{L_\infty(\cN)}\le K \right\} \subset P(\cM).
\]
By Zorn's Lemma, there exists a maximal element $e_0$ of $P_0$.
Let $e:= p_0-e_0$. Then for each nonzero projection $p\le e$, 
we have $e^{T(p)}(K,\infty)\ne 0$.

It follows from Lemma \ref{choosestepfunction} that there exists an element $$x=\sum\limits_{k\geq 1} \lambda_k p_k\in E(\mathcal{M}, \tau)\backslash \mathcal{M},$$
where $0<\lambda_k\in\mathbb{R}$ and $\{p_k\}_{k\geq 1}$ is a sequence of pairwise disjoint projections in $\cM$ satisfying $r(x)=\sum\limits_{k\ge 1}p_k\le e$. 
We claim that $T(x)\not\in \cN$.
Otherwise, since $T$ is normal and disjointness-preserving, it follows that
\[
\norm{T(x)}_{L_\infty(\cN)}=\norm{\sum\limits_{k\geq 1} \lambda_k T(p_k)}_{L_\infty(\cN)}
=\sup_{k\ge 1}\lambda_k\norm{T(p_k)}_{L_\infty(\cN)}
\stackrel{p_k\le e}{\ge}
K\sup_{k\ge 1}\lambda_k=\infty.
\]
which  is a contradiction. This   completes the proof.
\end{proof}

\begin{proof}[Proof of Theorem \ref{infinite}]
One can see from the 
  proof of Theorem \ref{infiniteTform} that  $w^*T(\cdot)$ is normal, where $w^*$ is a partial isometry. Without loss of generality, we may assume that $T$ is normal (in particular, positive).

If $E(\cM,\tau)\subset \cM$, then $x_0\in \cF(\tau) $ for any $x_0$ with $\tau$-finite support  $r(x_0)$.
Since $T$ is nonzero, it follows that $F(0,\infty)$ is nonzero. 
This together with the definition of Calkin spaces yields that $\mu(x_0)\in F(0,\infty )$. 

We only need to consider the case that $E(\cM,\tau)\not\subset  \cM$.
By Lemma \ref{notbounded}, there exists  $x_1\ge 0$ such that $r(x_1)$ is $\tau$-finite and   $T(x_1)\not \in \cN$.
Hence, $T(x_0+x_1)\ge T(x_1) \not \in \cN$. 
To show that $\mu(x_0)\in F(0,\infty)$, it suffices to show that  $\mu(x_0+x_1)\in F(0,\infty)$.
Hence, by replacing $x_0$ with $x_0+x_1$, we may assume that $T(x_0)\not\in \cN$. 

By \cite[Lemma 1.3]{CKS} (or Lemma~\ref{AandB}), we can find an abelian von Neumann subalgebra $\cA$ of $r(x_0)\cM r(x_0)$, containing all the spectral projections of $x_0$.
Noting that $\tau|_\cA$ is a finite faithful normal trace on $\cA$,
define
$$E(\cA,\tau|_\cA):=E(\cM,\tau)\cap S(\cA).$$
It is readily verified that $E(\cA,\tau|_\cA)$ is a Calkin operator space affiliated with $\cA $.

It follows from Theorem \ref{normal} that
\[
T|_{E(\cA,\tau|_\cA)}(x)=T(\mathbf{1}_\cA)J(x),
\quad  x\in E(\cA,\tau|_\cA),
\]
where $J:S(\cA,\tau|_\cA)\xrightarrow{\rm into} S(\cN,\nu)$ 
is a normal $^*$-homomorphism ($\cA$ is abelian, see Proposition \ref{Jcom}).
By Lemma \ref{isom}, there exists $p\in P(\cA)$ satisfying $J|_{\cA p}$ is a $^*$-isomorphism from $\cA p$ onto $J(\cA)$ and $J|_{\cA (\mathbf{1}_\cA-p)}=0$. 
Therefore,
\begin{align}\label{TAx}
    T|_{E(\cA,\tau|_\cA)}(x)=T(p)J(x),
\quad  x\in E(\cA,\tau|_\cA),
\end{align}
Moreover, arguing similarly as \eqref{T|_A}, we obtain that 
$T(p)$ commutes with $J(x)$ for each $x\in E(\cA,\tau|_\cA)$.

It follows from $T(x_0)\notin \cN$ that
there exists a constant $\lambda>0$ such that $$0<\nu\left(e^{T(x_0)}\left(\lambda,\infty\right)\right)\le \tau(r(x_0)) .$$
Let $q:=e^{T(x_0)}(\lambda,\infty)$ and define
\[
T_q(x):=T(x)q,\quad   x\in E(\cA,\tau|_\cA).
\]
For each $x\in E(\cA,\tau|_\cA)$, since $T(p)$ commutes with $J(x)$ and $S(\cA)$ is abelian (see \cite[Corollary 2.2.23]{DPS}), 
it follows from \eqref{TAx}  that
\[
T(x)T(x_0)=T(p)J(x)T(p)J(x_0)=T(p)J(x_0)T(p)J(x)=T(x_0)T(x).
\]
By \cite[Proposition 2.2.22]{DPS}, we have
\begin{align}\label{Tqform}
    T_q(x)=T(x)q=qT(x)q\in F(\cN,\nu)\cap S(\cN_q),
\quad  x\in E(\cA,\tau|_\cA).
\end{align}

It follows from $\cN$ is atomless that there exists a projection $q_0\ge q$ satisfying $\nu(q_0)=\tau(r(x_0))$. Note that 
$$F(\cN _{q_0},\nu|_{\cN _{q_0}}):=F(\cN,\nu)\cap S(\cN _{q_0})$$ 
is a Calkin operator space affiliated with $\cN _{q_0}$ and 
\[
T_q:E(\cA,\tau|_\cA)\xrightarrow{\rm into}F(\cN _{q_0},\nu|_{\cN _{q_0}})
\]
is a nontrivial ($T_q(x_0)\ne 0$) and normal (see \eqref{Tqform} and \cite[Proposition 2.2.25(iii)]{DPS}) mapping.
Let $x, y\in E(\cA,\tau|_\cA)$ be arbitrary such that $xy=0$.
Recall that  $q$ commutes with $T(x)$ and $T(y)$, which implies
\[
T_q(x)T_q(y)=qT(x)qT(y)q=qT(x)T(y)q=0.
\]
Therefore, $T_q$ is a nontrivial normal disjointness-preserving mapping.
By Theorem~\ref{main1}, we have 
$\mu(x_0)\in F(0,\nu(q_0))\subset F(0,\infty)$.
\end{proof}

\section{Positive isometries}

Let $\cM$ be a semifinite von Neumann algebra equipped with a semifinite faithful normal trace $\tau$. 
In this section, we obtain a necessary and sufficient condition for the orthogonality of positive measurable operators (see Theorem \ref{ab=0}).
Based on this, 
we establish the disjointness-preserving property of positive isometries on normed $\cM$-bimodules, which is the key to describe positive isometries on $\cM$-bimodules.

\subsection{Orthogonality of measurable operators}\label{Sec:ortho}

Let $p,q\in P(\cM)$ be orthogonal projections with $p+q=\mathbf 1$.
For any operator $x \in S(\cM)$, we write its block decomposition with respect to this splitting as
\[
x = \begin{pmatrix}
x_{pp} & x_{pq} \\
x_{qp} & x_{qq}
\end{pmatrix},
\]
where $x_{pp} = p x p$, $x_{pq} = p x q$, $x_{qp} = q x p$ and  $x_{qq} = q x q$.

The following lemma is well-known (see  \cite[Lemma~3.1]{PaulsenCB}   for the case of bounded operators). 
\begin{lemma}\label{Paulsen}
Let
\begin{align*}
x=
\begin{pmatrix}
p & c\\
c^* & q
\end{pmatrix}
\in S(\mathcal{M})
\end{align*}
be positive.
Then  $c\in\mathcal M$ and $\left\|c\right\|_{L_\infty(\cM)} \le 1$.
\end{lemma}

The following lemma is an infinite-dimensional version of \cite[Proposition 1.3.2]{Bha}.

\begin{lemma}\label{factorization}
Let
\begin{align*}
x=
\begin{pmatrix}
a & c\\
c^* & b
\end{pmatrix}
\in S(\mathcal{M})
\end{align*}
be positive.
Then there exists a contraction $d\in p\mathcal{M}q$ such that
$
c=a^{1/2}\, d\, b^{1/2}.
$ 
 In particular, \(l(c)\le l(a)\) and \(l(c^*)\le l(b)\).
\end{lemma}

\begin{proof}  
There exists a net of central projections $\{z_i\}_{i\in I}$ such that $\sum\limits_{i\in I}z_i=\mathbf{1}$ and $xz_i\in S(\cM,\tau)$ for each $i\in I$ (see, e.g., \cite[p. 2921]{AAK} or \cite[Lemma 3.3]{BHKS}).
Hence, without loss of generality, we may assume that $x\in S(\cM,\tau)$.

Firstly, we consider the case where
\[
a=pxp, b=qxq \in \cM.
\]
Fix $\varepsilon>0$ and define
\[
x_\varepsilon:=x+\varepsilon\mathbf{1}=
\begin{pmatrix}
a+\varepsilon p & c\\
c^* & b+\varepsilon q
\end{pmatrix}.
\]
Then, we have  $x_\varepsilon\ge0$ and both $a+\varepsilon p$ and $b+\varepsilon q$
are invertible in $p\cM p$ and $q\cM q$, respectively. 
Consider
\[
y_\varepsilon:=
\begin{pmatrix}
(a+\varepsilon p)^{-1/2} & 0\\
0 & (b+\varepsilon q)^{-1/2}
\end{pmatrix}
x_\varepsilon
\begin{pmatrix}
(a+\varepsilon p)^{-1/2} & 0\\
0 & (b+\varepsilon q)^{-1/2}
\end{pmatrix}.
\]
We have  $y_\varepsilon\ge0$ and
\[
y_\varepsilon=
\begin{pmatrix}
p & d_\varepsilon\\
d_\varepsilon^* & q
\end{pmatrix},
\qquad
d_\varepsilon=(a+\varepsilon p)^{-1/2}\, c\, (b+\varepsilon q)^{-1/2}.
\]
In particular,  for every $\varepsilon>0$, we have 
\[
c=(a+\varepsilon p)^{1/2}\, d_\varepsilon\, (b+\varepsilon q)^{1/2}\in \cM. 
\]
By Lemma~\ref{Paulsen}, the positivity of $y_\varepsilon$ implies that 
$
\left\|d_\varepsilon \right\|_{L_\infty(\cM)}\le1
$
for every $\varepsilon>0$.

Since the unit ball of $\mathcal M$ is  compact in the weak operator topology, it follows that there exists
a  net $\{d_{\varepsilon_i}\}_i$ such that $\varepsilon_i \downarrow 0$ and 
$d_{\varepsilon_i}\to d$ in the weak operator topology, with $\left\|d\right\|_{L_\infty(\cM)}\le1$.
Moreover,
\[
\norm{(a+\varepsilon_i p)^{1/2}- a^{1/2}}_{L_\infty(\cM)}\to 0,\quad
\norm{(b+\varepsilon_i q)^{1/2}- b^{1/2}}_{L_\infty(\cM)}\to 0.
\]
Passing to the limit yields 
\begin{align*}
    (a+\varepsilon_i p)^{1/2}\, d_{\varepsilon_i}\, (b+\varepsilon_i q)^{1/2}
    \to a^{1/2}\, d\, b^{1/2}
\end{align*}
in the weak operator topology, 
that is,
$$c=a^{1/2}\, d\, b^{1/2}.$$
Without loss of generality, we may assume that $l(d)\le r(a)$ and $r(d)\le l(b)$.
In particular, $d$ is unique in this sense.

Now, we consider the general case.
For each fixed $n \ge 1$, set
\[
  p_n = e^{a}[0,n], \qquad q_n = e^{b}[0,n].
\]
Passing to the reduced von Neumann algebra $(p_n+q_n)\cM(p_n+q_n)$ and replacing $\{p, q\}$ with $\{p_n,q_n\}$,  we may denote
\[
  a_n := p_n a p_n\in\cM, \qquad b_n := q_n b q_n \in \cM.
\]
Then we have 
$$(p_n+q_n)x(p_n+q_n)=
\begin{pmatrix}
    a_n & p_ncq_n\\
    q_nc^*p_n & b_n
\end{pmatrix}.$$
By the previously treated case, there exists a contraction $d_n \in p_n\cM q_n$ such that
\begin{align}\label{limpcq}
    p_n c q_n = a_n^{1/2} d_n b_n^{1/2}.
\end{align}
For $n \le m$, it follows from $p_n \le p_m$ and $q_n \le q_m$ that 
\[
  p_n d_m q_n = d_n.
\]
Noting that $\tau(p_n-p)\to 0$ and $\tau(q_n- q)\to 0$ 
  as $n\to \infty $,  it follows from \cite[Proposition 2.6.11]{DPS} that 
\begin{align}\label{c converge}\begin{split}
 p_n c q_n \xrightarrow{t_m} c 
\end{split}
\end{align}
as $n \to \infty $. 
By Lemma \ref{solim}, there exists $d\in \cM$ such that $d=pdq =so-\lim_n d_n $ and $p_n d q_n =d_n $ for every $n\ge 1$. 
Since $\norm{d_n}_{L_\infty (\cM)}\le 1$, it follows that $\norm{d}_{L_\infty (\cM)}\le 1$. 
Arguing similarly to \eqref{c converge}, we obtain 
$$d_n=p_n d q_n \xrightarrow{t_{m}} d$$
as $n\to \infty$. 
Since $a_n^{1/2}=a^{1/2}p_n 
$ and $b_n^{1/2}=q_nb^{1/2} 
$ \cite[Proposition 2.2.22]{DPS}, it follows  that
\[
c\xleftarrow{t_{m}}p_ncq_n\stackrel{\eqref{limpcq}}{=}a_n^{1/2} d_n b_n^{1/2} =a^{1/2} p_n d q_n b^{1/2}
\xrightarrow[\tiny \mbox{\cite[Prop. 2.6.11]{DPS}}]{t_{m}} a^{1/2} d b^{1/2}, 
\]
i.e., $c = a^{1/2} d b^{1/2}$.
This completes the proof.
\end{proof}

The following lemma generalizes \cite[Proposition 4.5]{HSZ20} with a substantially simpler proof. 

\begin{lemma}\label{lm1}
Let $x,y \in S(\cM)_+$.
If $xy + yx = 0$, then $xy = 0$.
\end{lemma}

\begin{proof} 
By multiplying the equation $xy+yx=0$ on the left and on the right by $x$,
respectively, we obtain that
\[
x^2y+xyx=0=xyx+yx^2,
\]
which implies $x^2 y = y x^2$.
By \cite[Proposition~2.2.22]{DPS} (choosing $f(t) = \sqrt{t}$, $g(t)=t$, $t\ge0$), we have
$xy=\left(x^2\right)^\frac{1}{2}y=y\left(x^2\right)^\frac{1}{2}=yx$.
Therefore,
\[
0 = xy + yx = 2xy,
\]
i.e., $xy = 0$.
\end{proof}

\begin{theorem}\label{ab=0}
    Suppose that $\cM$ is a semifinite von Neumann algebra.
Let $a,b\in S(\cM)_+$. Then the following assertions are equivalent.
\begin{itemize}
    \item[(1)] There exists a projection $p\in P(\cM)$ such that $p$ commutes with $a-b$ and 
\begin{align}\label{absa-b}
|a-b|=p(a+b)p+(\mathbf{1}-p)(a+b)(\mathbf{1}-p).
\end{align}
\item[(2)] $ab=0$.
\end{itemize}
\end{theorem}

\begin{proof} (2)$\Rightarrow$(1). Suppose that  $ab=0$. 
The implication follows by setting 
  $p=r(a)$.

(1)$\Rightarrow$(2). 
Suppose that condition (1) holds.
Denote $q:=\mathbf{1}-p$.
Since $p$ commutes with $a-b$, it follows that
\begin{align}\label{a-b}
    a-b=p(a-b)p+q(a-b)q=
    \begin{pmatrix}
        a_{pp}-b_{pp} & 0\\
        0 & a_{qq}-b_{qq}
    \end{pmatrix},
\end{align}
which implies that
\begin{align}\label{abconcide}
    a_{qp}^*=a_{pq}=b_{pq}=b_{qp}^*.
\end{align}
It follows from \eqref{absa-b} that
\[
|a-b|=
\begin{pmatrix}
    a_{pp}+b_{pp} & 0\\
    0 & a_{qq}+b_{qq}
\end{pmatrix},
\]
which together with \eqref{a-b} yields that $(a_{pp}+b_{pp})^2=(a_{pp}-b_{pp})^2$, i.e.,
$a_{pp}b_{pp}+b_{pp}a_{pp}=0$.
Similarly, $a_{qq}b_{qq}+b_{qq}a_{qq}=0$.
By Lemma \ref{lm1}, we have 
\begin{align}\label{disjoint}
    a_{pp}b_{pp}=a_{qq}b_{qq}=0.
\end{align}

Applying Lemma~\ref{factorization} to $a, b\ge 0$, we obtain that 
\begin{align*}
l(a_{pq}) \le  l(a_{pp}),~  r(a_{pq})\le r(a_{qq}),~ 
l(b_{pq}) \le  l(b_{pp}),~ r(b_{pq})\le r(b_{qq}),
\end{align*}
which together with \eqref{disjoint} yields that $l(a_{pq})l(b_{pq})=r(a_{pq})r(b_{pq})=0$.
By \eqref{abconcide}, we have
\[
a_{pq}=b_{pq}=a_{qp}=b_{qp}=0.
\]
Thus, 
\[
a=pap+qaq,\quad\, b=pbp+qbq,
\]
which together with \eqref{disjoint} yields that $ab=0$.
This completes  the proof.
\end{proof}

\begin{remark}
    The conclusions of Theorem \ref{ab=0},  Lemmas \ref{factorization} and \ref{lm1} remain valid for locally measurable operators affiliated with a  von Neumann algebra, see \cite{BCLSZ} for the definition  of locally measurable operators.  
    Therefore, the main result of this section, Theorem  \ref{predisjoint}, remains valid for bimodules of  measurable operators (and more generally,  bimodules of locally measurable  operators, see \cite[Section 6]{BCS} for the definition). 
\end{remark}

\subsection{Proof of Theorem \ref{predisjoint} and its applications}

A norm $\norm{ \cdot }_\cE$ on a  normed space \( \mathcal{E} \subseteq S(\mathcal{M}, \tau) \) is said to be {\it strictly monotone} if for any $x_1, x_2 \in \cE$,
$
0 \leq x_1 \leq x_2 
$
and $x_1\neq x_2$
implies
\[
\norm{x_1}_{\mathcal{E}} < \norm{x_2}_{\mathcal{E}}.
\]

\begin{proof}[Proof of Theorem \ref{predisjoint}]
  Let $0\le x,y\in E(\cM,\tau)$ be arbitrary such that $xy=0$.
	Observe that $$-T(x)-T(y)\le T(x)-T(y) \le T(x)+T(y) .$$
	By \cite[Theorem 1]{Bikchentaev12} (see also the proof of \cite[Theorem 2]{SV}), 
    the projection $p:= e^{T(x)-T(y)}[0,\infty )$
    satisfies that 
	$$ 
	2 |T(x) -T(y)| \le  T(x)+T(y)   + u (T(x)+T(y) )u,  
	$$
	where $u:=2p-\mathbf{1}\in U(\cN)$.
    By the triangular inequality of  $\norm{\cdot}_F$  and the isometric property of $T$, we have 
	\begin{align*}
		2  \norm{T(x) -T(y)} _F&\le \norm{ T(x)+T(y)   + u (T(x)+T(y) )u} _F\\
		&\le \norm{ T(x)+T(y)   } _F +\norm{ u (T(x)+T(y) )u} _F\\
		&\leq 2\norm{ T(x)+T(y)}_F=2\norm{x+y}_E\\
        &=2\norm{x-y}_E=2\norm{T(x)-T(y)}_F.
	\end{align*}
	Hence, 
	we have 
	$$2  \norm{T(x) -T(y)} _F= \norm{ T(x)+T(y)   + u (T(x)+T(y) )u} _F. $$
	By the definition of a strictly monotone norm, we obtain that 
	$$ 2 |T(x) -T(y)| =  T(x)+T(y)   + u (T(x)+T(y) )u.  $$
Moreover, it follows from $u=2p-\mathbf{1}$ that
\begin{align*}
|T(x) -T(y)| & =  T(x)+T(y) -p(T(x)+T(y)) \\
& \quad -(T(x)+T(y))p  + 2p(T(x)+T(y) )p \\
& =p(T(x)+T(y))p+(\mathbf{1}-p)(T(x)+T(y))(\mathbf{1}-p).
\end{align*}
Noting that $p$ commutes with $T(x)-T(y)$, it follows from Theorem~\ref{ab=0} that 
$T(x)T(y)=0$, i.e., $T$ preserves disjointness.
\end{proof}


The following result is an immediate consequence of Theorem \ref{predisjoint}, Theorem~\ref{normal}  and Remark \ref{Jsubalg}, which extends the main result in \cite{SV}.

\begin{theorem}\label{isoform}
Suppose that $\cM$ is a von Neumann algebra with a finite faithful normal trace $\tau$ and $\cN$ is a semifinite von Neumann algebra with a semifinite faithful normal trace $\nu$.
Let $E(\cM,\tau)$ and $F(\cN,\nu)$
be a normed $\cM$- and a normed $\cN$-bimodule, respectively, and $\norm{\cdot}_F$ is strictly monotone.
    If $T$ is a positive isometry from $E(\cM,\tau)$ into $F(\cN,\nu)$, then $T$ has the following form
    \[
    T(x)=T(\mathbf{1})J(x)=J(x)T(\mathbf{1}), \quad  x\in \cM,
    \]
    where $J$ is a  normal  Jordan $^*$-monomorphism from $\cM$ into $\cN$.
    
    Moreover, if $\nu(J(\mathbf{1}))<\infty$ and $T$ is normal, then $J$ can be extended to a normal Jordan $^*$-monomorphism $J:S(\cM,\tau)\to S(\cN,\nu)$ satisfying
    \[
    T(x)=T(\mathbf{1})J(x)=J(x)T(\mathbf{1}), \quad  x\in E(\cM,\tau).
    \]

\end{theorem}

\begin{remark}
\leavevmode
\begin{enumerate}
\item 
Positive disjointness-preserving isometries between Banach function spaces (not necessarily surjective) are order monomorphisms. Hence, any such isometry is normal (equivalently, is order continuous \cite[Lemma 1.24]{AA}).
Observe that the operator 
 $T$ in Theorem \ref{isoform}   is normal  on $\cM$ by the normality of $J$ (see \cite[Proposition 2.2.22 and Proposition 2.2.25(iii)]{DPS}).
 However,  we cannot determine whether $T$ is automatically normal on the space $E(\cM,\tau)$ even when we have the disjointness-preserving property.
\item If $(\cM,\tau)$ in Theorem \ref{isoform} is semifinite and $T$ is normal, then by Theorems~\ref{predisjoint} and  \ref{infiniteTform}, T has the form
\[
     T(x)=bJ(x), \quad x\in E(\cM,\tau)\cap \cM,
    \]
     where 
     $b$ is a (possibly not measurable) positive self-adjoint operator affiliated with $\cN$ and
     $J:\cM\to \cN$ is a normal Jordan $^*$-homomorphism. 
     {\color{blue}Moreover, $e^b(\delta)\in Z(J(\cM))$ for all $\delta\subset\mathbb{R}$.}
     This extends \cite[Corollary 4.9]{HSZ20} in the  settings of Banach symmetric spaces.  
\end{enumerate}
\end{remark}

 \label{definition normed Calkin}
Recall that  a Calkin space $E(\cM,\tau)$ is an $\cM$-bimodule (see Section \ref{s:symmetric}). 
We call  $E(\cM,\tau)$ a normed Calkin $\cM$-bimodule if $E(\cM,\tau)$ is additionally a normed $\cM$-bimodule, i.e., $E(\cM,\tau)$  is equipped with a norm $\norm{\cdot}_E$ satisfying 
\[
\norm{uxv}_E\le \norm{u}_{L_\infty(\cM)}\norm{v}_{L_\infty(\cM)}\norm{x}_E,
\quad x\in E(\cM,\tau),\, u,v\in \cM.
\]

\begin{theorem}\label{contain}
Let $\cM$ be an atomless von Neumann algebra equipped with finite faithful normal trace $\tau$.
    Suppose that $E(\cM,\tau)$ and $F(\cM,\tau)$ are normed Calkin $\cM$-bimodules 
    (for example, symmetrically normed spaces), and 
      $\norm{\cdot}_F$ is strictly monotone.
    If 
	 $T$ is a positive isometry from $E(\mathcal{M}, \tau)$ into $F(\mathcal{M}, \tau) $,
	then 
		\[
	E(\mathcal{M}, \tau) \subseteq F(\mathcal{M}, \tau).
	\] 
	Furthermore, if $T$ is surjective and $E(\cM,\tau)$ has strictly monotone norms, then 
	\[
	E(\mathcal{M}, \tau) = F(\mathcal{M}, \tau).
	\]
\end{theorem}

\begin{proof}
By Theorem \ref{predisjoint},  $T$ is disjointness-preserving.
It follows from Theorem \ref{isoform} that
\begin{align}\label{T_1M}
T(x)=T(\mathbf{1})J(x)=J(x)T(\mathbf{1}), \quad  x\in \cM,
\end{align}
    where $J$ is a normal Jordan $^*$-monomorphism from $\cM$ into $\cM$.
By Proposition \ref{Jnormea}, we conclude that  $J$ is continuous in the measure topology.
Define
\[
T_1(x):=J(\mathbf{1})T(x)J(\mathbf{1}),\quad  x\in E(\cM,\tau).
\]
For each $x\in \cM$, by \eqref{T_1M}, we have
\begin{equation}\label{T1=T}
\begin{aligned}
    T_1(x)&=J(\mathbf{1})T(\mathbf{1})J(x)J(\mathbf{1})=T(\mathbf{1})J(x)J(\mathbf{1})\\
    &=J(x)T(\mathbf{1})J(\mathbf{1})=J(x)T(\mathbf{1})=T(x).
\end{aligned}
\end{equation}
We claim that $T_1$ is normal on $E(\cM ,\tau)$. Indeed, let $\{x_i\}_{i\in I}\subset E(\cM,\tau)_+$ be an arbitrary decreasing net satisfying $x_i\downarrow 0 $.
It follows from \cite[Theorem 2.6.3]{DPS} that $x_i\xrightarrow{ t_m} 0$.

Fix a constant $K>0$. 
Observing that $x_ie^{x_i}(0,K]\xrightarrow{ t_m}_i  0$ \cite[Proposition 2.6.1(iv)]{DPS}, it follows from the continuity of $J$ in measure and \cite[Proposition 2.6.11(i)]{DPS} that 
\begin{align}\label{boundedm}
    T_1\left(x_ie^{x_i}(0,K]\right)\stackrel{\eqref{T1=T}\eqref{T_1M}}{=}
    T(\mathbf{1})J(x_ie^{x_i}(0,K])\xrightarrow{t_m}0.
\end{align}
Noting that $e^{x_i}(0,K]x_ie^{x_i}(K,\infty)=0$, since $T$ preserves disjointness, it follows from \eqref{upjq} that $J(e^{x_i}(0,K])=r(T(e^{x_i}(0,K]))$, which implies that 
\begin{align}\label{org}
    J(e^{x_i}(0,K])T(x_ie^{x_i}(K,\infty))=0=T(x_ie^{x_i}(K,\infty))J(e^{x_i}(0,K]).
\end{align}
Since $\mathbf{1}-e^{x_i}(0,K]=e^{x_i}(K,\infty)\xrightarrow{t_m}_i 0$ and $J$ is continuous on $\cM$ in the measure topology,
it follows that $\tau(J(\mathbf{1}-e^{x_i}(0,K]))\to_i 0$.
By the definition of measure topology, we have
\begin{align*}
    &T_1(x_ie^{x_i}(K,\infty))=J(\mathbf{1})T(x_ie^{x_i}(K,\infty))J(\mathbf{1})\\
    \stackrel{\eqref{org}}{=}&
    J(\mathbf{1}-e^{x_i}(0,K])T(x_ie^{x_i}(K,\infty))J(\mathbf{1}-e^{x_i}(0,K])\xrightarrow{t_m} 0,
\end{align*}
which together with \eqref{boundedm} yields that
\[
T_1(x_i)=T_1\left(x_ie^{x_i}(0,K]\right)+T_1(x_ie^{x_i}(K,\infty))\xrightarrow{t_m} 0.
\]
Observing that $T_1$ is positive \cite[Proposition 2.2.24(iv)]{DPS}, we have $T_1(x_i)\downarrow 0$ \cite[Proposition 2.6.1(iii)]{DPS}, which proves our claim.

Let $x,y\in E(\cM,\tau)_+$ be arbitrary such that $xy=0$. 
For each $k\ge1$, 
denote $x_k:=xe^x(0,k)\in \cM_+$.
Since $T$ is disjointness-preserving and $x_ky=0$, 
it follows that for each $k\ge 1$,
\begin{equation}\label{finitecom}
\begin{aligned}
    T_1(x_k)T_1(y)& = T(x_k)J(\mathbf{1})T(y)J(\mathbf{1})\\
    &=T_1(x_k) T(y)J(\mathbf{1})
    \stackrel{\eqref{T1=T}}{=}T(x_k)T(y)J(\mathbf{1})=0.
\end{aligned}
\end{equation}
By the normality of $T_1$, $x_k\uparrow x$ implies $T_1(x_k)\uparrow T_1(x)$.
Then we have
    $T_1(x_k)\xrightarrow{ t_m}T_1(x)$ \cite[Theorem 2.6.3]{DPS},
    which together with \cite[Proposition 2.6.11]{DPS} yields that
    \[
    0\stackrel{\eqref{finitecom}}{=}T_1(x_k)T_1(y)\xrightarrow{t_m}T_1(x)T_1(y)=0.
    \]
Consequently,
    \[
    T_1:E(\cM,\tau)\xrightarrow{\rm into} F(\cM,\tau)
    \]
    is a nonzero ($T_1({\bf 1}) =T({\bf 1})\ne 0$)  normal disjointness-preserving mapping. 
	It follows from Theorem \ref{main1} 
	that 
		\begin{align}\label{E and F}
	E(\mathcal{M}, \tau) \subseteq F(\mathcal{M}, \tau).
	\end{align}

Furthermore, if $T$ is surjective and $E(\cM,\tau)$ has strictly monotone norms, then by \cite[Lemma 3.1]{dC20}, $T$ and $T^{-1}$ are order isomorphisms and normal.
It follows from Theorem \ref{predisjoint} that $T^{-1}$ is a normal disjointness-preserving isometry.
Hence, by \eqref{E and F}, the 
proof is complete.
\end{proof}

\begin{remark}
\leavevmode
\begin{enumerate}
\item Suppose that $E(\cM,\tau)$ and $F(\cM,\tau)$ are strongly symmetric operator spaces (in the sense of Lindenstrauss and Tzafriri) affiliated with an atomless finite von Neumann algebra equipped with a faithful normal tracial state.  It is shown in~\cite{FGHS} that if $E(\cM,\tau)$ and $F(\cM,\tau)$ are isometric (not necessarily positively), then $E(\cM,\tau)=F(\cM,\tau)$ as sets. Moreover, if $E(\cM,\tau) \ne L_p(\cM,\tau)$, then $\norm{\cdot}_E =\lambda\norm{\cdot}_F$ for some positive number $\lambda $. 
\item 
In 1970, Mityagin asked whether a symmetric function space $E(0,1 )\ne L_p(0,1)$ can be isometric to any symmetric function space $F(0,\infty)$\cite{Mityagin}. 
This question was recently answered in the negative \cite[Theorem 1.5.1]{FGHS}. 
However, this question has an affirmative answer for general function spaces. 
For example, let $\sigma $ be a transformation from $(0,\infty )$ onto $(0,1)$ and let $F(0,\infty )$ be a Banach function space. 
The function space 
$$E(0,1) = \{  f:  ~f\circ \sigma  \in F(0,\infty )\}$$
equipped with the norm $\norm{f}_E:=\norm{f\circ \sigma }_F$ is positively isometric to $F(0,\infty)$. 
    \item Suppose that $(\cM,\tau)$ in Theorem \ref{contain} is semifinite and $T$ is additionally normal.
By Theorem \ref{predisjoint} and Theorem \ref{infinite}, for each $\tau$-finite $p\in P(\cM)$, we have
\[
E(\cM_p,\tau|_{\cM_p}):=\left\{pxp:x\in E(\cM,\tau)\right\}\subset F(\cM,\tau).
\]
\end{enumerate}

\end{remark}

	{\bf Acknowledgments:} The authors would like to thank Professor Marat Pliev for helpful discussions.

{\bf Data Availability}
Data sharing not applicable to this article as no datasets were generated or analysed during the current study. 

{\bf  Conflict of interest} On behalf of all authors, the corresponding author states that there is no conflict of interest.

{\bf Ethical statement}  This manuscript complies to the Ethical Rules applicable for this journal. 

\bibliographystyle{amsalpha}

\end{document}